\documentclass{amsart}

\usepackage[utf8]{inputenc}
\usepackage{fullpage}
\usepackage{bbm,amssymb,nicefrac,tikz-cd,tikz}
\usepackage{kbordermatrix} 
\usepackage[colorlinks=true,citecolor=blue,backref]{hyperref}
\usepackage[nameinlink]{cleveref}

\newcommand{\set}[1]{\left\{ #1 \right\}}
\newcommand{\R}{\mathbb{R}}

\newcommand{\Z}{\mathcal{Z}}
\newcommand{\inner}[2][]{\left\langle #2 \right\rangle_{#1}}
\newcommand{\paren}[1]{\left(#1\right)}
\newcommand{\abs}[1]{\left| #1 \right|}
\newcommand{\size}[1]{\left| #1 \right|}
\newcommand{\norm}[2][]{\left\| #2 \right\|_{#1}}
\newcommand{\bigOh}[1]{\mathcal{O}\!\paren{#1}}
\newcommand{\bigTheta}[1]{\Theta\!\paren{#1}}
\newcommand{\expect}[1]{\mathbb{E}\left[#1\right]}

\newcommand{\hide}[1]{}

\newcommand{\IPL}[1][]{\mathcal{L}^{\textrm{IP}}_{#1}}
\newcommand{\sHL}{\mathcal{L}^{sH}}
\newcommand{\NP}{$\mathcal{NP}$}
\newcommand{\NPc}{$\mathcal{NP}$-complete}
\DeclareMathOperator{\Ker}{Ker}
\DeclareMathOperator{\IM}{Im}

\newcommand{\half}{\nicefrac{1}{2}}
\newcommand{\one}{\mathbbm{1}}
\newcommand{\boundary}{\partial} 

\DeclareMathOperator{\sgn}{sgn}
\DeclareMathOperator{\Vol}{Vol}
\DeclareMathOperator{\trace}{trace}
\DeclareMathOperator{\Cor}{Cor}
\DeclareMathOperator*{\argmin}{argmin}

\theoremstyle{plain}
\newtheorem{theorem}{Theorem}
\newtheorem{lemma}{Lemma}
\newtheorem{corollary}{Corollary}
\theoremstyle{definition}
\newtheorem*{define}{Definition}
\newtheorem*{mEML}{Matrix Expander Mixing Lemma}
\newtheorem*{cEML}{Combinatorial Laplacian Expander Mixing Lemma}
\newtheorem*{nEML}{Normalized Laplacian Expander Mixing Lemma}

\newcommand{\sinan}[1]{\color{brown}#1 - Sinan\ \color{black}}

\usepackage[normalem]{ulem}

\author{Sinan G. Aksoy and Stephen J.\ Young}
\thanks{This research was partially supported by the Mathematics for Artificial Reasoning in Science
(MARS) initiative, under the Laboratory Directed Research and Development (LDRD) Program
at Pacific Northwest National Laboratory (PNNL). PNNL is a multi-program national laboratory
operated for the U.S. Department of Energy (DOE) by Battelle Memorial Institute under Contract
No. DE-AC05-76RL01830. \emph{PNNL Information Release:} PNNL-SA-210426}

\title{Re-imagining Spectral Graph Theory}

\begin{document}

\allowdisplaybreaks
\begin{abstract}
We propose a Laplacian based on general inner product spaces, which we call the {\it inner product Laplacian}. We show the combinatorial and normalized graph Laplacians, as well as other Laplacians for hypergraphs and directed graphs, are special cases of the inner product Laplacian. After developing the necessary basic theory for the inner product Laplacian, we establish generalized analogs of key isoperimetric inequalities, including the Cheeger inequality and expander mixing lemma. Dirichlet and Neumann subgraph eigenvalues may also be recovered as appropriate limit points of a sequence of inner product Laplacians. In addition to suggesting a new context through which to examine existing Laplacians, this generalized framework is also flexible in applications: through choice of an inner product on the vertices and edges of a graph, the inner product Laplacian naturally encodes both combinatorial structure and domain-knowledge.
\end{abstract}
\maketitle

\section{Introduction}
As diverse as spectral graph theory’s connections are to different areas in mathematics, so too are its applications. 
With well-established connections to Riemannian geometry, algebraic topology, probability and more, modern spectral theory has evolved from its origins in molecular chemistry to become a powerful tool in far-ranging scientific domains. 
This depth in both the theory and practice of graph eigenvalues has spurred an overwhelming number of proposed Laplacian matrices for studying graphs and related combinatorial structures. These Laplacians vary in the mathematical traditions that inspired them (e.g. Hodge theory \cite{lim2020hodge, schaub2020random} to Markov chains \cite{chung1997spectral, li2010random}), the properties of graphs their spectra can detect (e.g. bipartiteness \cite{butler2010note}), and the types of combinatorial structures to which they may be applied (e.g. directed graphs \cite{chung2005laplacians, li2012digraph} to hypergraphs \cite{chitra2019random, chung1992laplacian, hayashi2020hypergraph, zhou2006learning}). Consequently, there is a need for Laplacian generalizations which unify existing Laplacians, and provide a general framework for their study and application. 

A key challenge in defining generalized Laplacians is striking a balance between providing sufficient generality, while imposing enough structure to enable a non-trivial theory. For example, a definition \cite[Sec. 1.1]{cohen2016faster} that arguably errs on the side of the satisfying the former views a generic Laplacian as ``simply a matrix with non-positive off-diagonal entries such that
each diagonal entry is equal to the sum of the absolute value of the other off-diagonal entries in that
column”. While this generality is advantageous in certain algorithmic contexts (e.g. as in \cite{cohen2016faster}), it is perhaps too general to mathematically unite different Laplacians in a meaningful way, or place them in a new context. On the other hand, a definition that arguably errs on the side of imposing structure to enable theory-building comes from \cite{Horak:WeightedLaplcian}. There, Horak and Jost use Hodge theory to define a Laplacian framework for simplicial complexes, which recovers the combinatorial and normalized Laplacians as special cases. However, their framework imposes a strong assumption -- an orthogonality condition between all faces of the simplicial complex. This limits this Laplacian's utility in practice, where imposing an orthogonality requirement between entities within social, biological, cyber, or other networks may be wholly unrealistic or at odds with domain knowledge.

In this work, we propose a Laplacian framework using general inner product spaces. We call our Laplacian generalization the ``inner product Laplacian". We take an approach rooted in simplicial homology like Horak and Jost, but do not impose any orthogality requirement. To derive a spectral theory for the inner product Laplacian requires that we invoke properties of inner product spaces in our analyses, such as {\it conformality} (\Cref{S:conformal}), a measure controlling the relationship between an inner product space and a specified basis for the same inner product space, whose computation we show is NP-complete. While these inner product space concepts and related tools may be of independent interest, we invoke them to derive generalized versions of fundamental results in spectral graph theory. Our main results include:
\begin{itemize}
    \item a formalization of how inner product Laplacian generalizes existing Laplacians, such as normalized and combinatorial Laplacian, as well as a host of Laplacians for hypergraphs and directed graphs (see \Cref{L:hyper} and \Cref{L:digraph_lap}, as well as \Cref{SS:higher_order}),
    \item proof that subgraph eigenvalues such as the Neumann eigenvalues, can be recovered as limit points of a sequence of inner product Laplacians (\Cref{T:neumann}),
    \item upper and lower bounds of Cheeger inequalities for the inner product Laplacian (\Cref{T:cheeger}), and
    \item a discrepancy inequality, the expander mixing lemma for the inner product Laplacian (\Cref{T:EML}).
\end{itemize}

While our approach suggests a rich theory underlying the inner product Laplacian, we've proposed the inner product Laplacian with hopes of enabling more domain-faithful and flexible spectral analyses of network data in practice. In particular, the inner product Laplacian provides a principled manner of fusing both combinatorial and non-combinatorial data, without resorting to ad-hoc solutions such as domain-informed edge-weighting schemes. Instead, domain-knowledge may be encoded as appropriate inner product spaces on the vertices and edges, which are then used to construct the inner product Laplacian. In this way, the inner product Laplacian combines both combinatorial, graphical information about adjacent vertices, as well as non-combinatorial information regarding the similarity of {\it non-adjacent} vertices and arbitrary pairs of edges, in a single Laplacian. Recent, concurrent work by the authors and others \cite{aksoy2024unifying} demonstrates the potential efficacy of this approach: defining domain-inspired inner product spaces for molecular data, they apply the resulting inner product Laplacian within a graph neural network framework to estimate the potential energy of atomistic configurations and find that an inner product Laplacian approach outperforms analogous, existing state-of-the-art methods. While far more work is necessary to understand the inner product Laplacian's potential in practice, this preliminary evidence suggests there are contexts in which the inner product Laplacian's fusion of domain data and combinatorial structure is effective, and more than ``the sum of its parts". 

We organize our work as follows. In \Cref{S:survey}, we survey existing Laplacians as belonging to three broad categories, and place the inner product Laplacian within this context. \Cref{S:conformal} introduces and analyzes properties of inner product spaces relevant to the inner product Laplacian, including strong and weak conformality. These concepts and the associated lemmas will later prove essential in establishing our results. Having established the necessary background, \Cref{S:IPL} formally defines and discusses the inner product Laplacian. Here, \Cref{SS:higher_order} formalizes the relationship between the inner product Laplacian and existing Laplacians for graphs, directed graphs, and hypergraphs; while \Cref{SS:subgraphs} also establishes a connection between the inner product Laplacian and subgraph eigenvalues, showing that Neumann eigenvalues can be recovered as the limit of a the spectral of a sequence of inner product Laplacians. \Cref{S:iso} then proves canonical isoperimetric inequalities for the inner product Laplacian, including the Cheeger inequality and the expander mixing lemma.

\section{Survey of Graph Laplacians}\label{S:survey}
Since Kirchoff's matrix tree theorem \cite{kirkoff_orig}, a dizzyingly large literature has emerged on graph Laplacians. Named for its correspondence to the discrete Laplace operator, the term ``graph Laplacian" is most commonly understood to mean the combinatorial Laplacian: the diagonal matrix of degrees minus the adjacency matrix, $D-A$. Yet, the proliferation of proposed Laplacians is now such that, as put by Von Luxburg \cite{von2007tutorial}: ``there is no unique convention which matrix exactly is called `graph Laplacian'...every author just calls `his' matrix the graph Laplacian." We argue the modern myriad of graph Laplacians can be organized into roughly three, potentially overlapping categories. First are what can be called {\it Laplacian variants}: these ``tweak" the definition of the combinatorial Laplacian, such as through sign changes, normalization, or averaging. In doing so, these new Laplacian variants often capture different structural properties of the graph. Second, ``graph-generalization Laplacians" are those which try to extend a given Laplacian to a more general class of combinatorial objects than graphs. Here, the goal is to incorporate additional information, such as edge-weights or directionality, in a principled and useful manner. Third, there are ``Laplacian generalizations", which cast multiple Laplacian variants as special cases of a single, more general definition. These aim to explain the differences exhibited by Laplacians through a unifying framework. Below, we review existing literature within each of these three categories, as well as place our work among them. \\

\paragraph{\bf Laplacian variants} In addition to the combinatorial Laplacian, two of the most often studied ``Laplacian variants" are the normalized Laplacian and signless Laplacian. As popularized by Chung's monograph \cite{chung1997spectral}, the normalized Laplacian can be written in terms of the combinatorial Laplacian $L=D-A$ as
\[
\mathcal{L}=D^{-1/2}LD^{-1/2}.
\]
 In large part, this Laplacian captures different properties than $L$ because its spectra is intimately related to stochastic processes. Indeed, conjugating $\mathcal{L}$ by $D^{1/2}$ yields $I-D^{-1}A$, a matrix that is sometimes recognized as the ``random walk Laplacian" because $D^{-1}A$ represents the transition probabilities in a uniform random walk. A matrix more subtly different to the combinatorial Laplacian is the signless Laplacian \cite{cvetkovic2005signless,haemers2004enumeration},
\[
|L|=D+A.
\]
Although sharing in many of the same properties as $L$, the signless Laplacian also differs in basic ways. For instance, the spectra of $|L|$ can always identify the number of bipartite components whereas that of $L$ cannot. In general, although differences between the combinatorial, signless, and normalized Laplacians abound, they also share important similarities. All three are positive semi-definite, and are essentially equivalent for regular graphs, in which case their spectra are related to one another by trivial shifts and scalings. \\

\paragraph{\bf Graph-generalization Laplacians} The largest class of proposed Laplacian matrices are  ``graph-generalization Laplacians". Unlike the aforementioned variants which focus on tweaking the abstract notion of a Laplacian, these aim to {\it extend} the notion of (combintorial, signless, or normalized) Laplacian to richer types of combinatorial objects than the usual (undirected, unweighted, simple) graph. In some cases, this generalization is straightforward. For instance, edge-weighted variants may be obtained by replacing the unweighted adjacency matrix with a weighted one; vertex-weighted variants also may be defined, albeit with more care \cite{chung1996combinatorial}. In other cases, however, there is no obvious or canonical extension, yielding a rich literature weighing different approaches. For instance, proposals for directed graph Laplacians are marked by how they navigate the tradeoff between faithfully encoding directionality while preserving matrix properties such as symmetry. Chung \cite{chung2005laplacians, chung2006diameter}, choosing the latter route, proposed a symmetrized digraph Laplacian
\[
L=\Pi - \frac{\Pi P + P^T \Pi}{2}
\]
where the diagonal stationary distribution matrix $\Pi$  plays the role of degree matrix, and the symmetrization of the transition matrix $P$ plays that of the adjacency matrix. Rather than symmetrize, some have opted to study asymmetric digraph Laplacians \cite{li2012digraph}, while others use complex numbers to encode directionality in a matrix that is Hermitian but not symmetric, as in Mohar's proposed adjacency matrix \cite{mohar2020new}, and the closely related magnetic Laplacian \cite{zhang2021magnet}. In a similar vein, a variety of Laplacians for hypergraphs and simplicial complexes have been recently proposed and studied, ranging from Markov-chain based Laplacians \cite{benko2024hypermagnet, chitra2019random, hayashi2020hypergraph, lu2013high, zhou2006learning} to so-called hypermatrices \cite{cooper2012spectra} or tensor Laplacians \cite{aksoy2024scalable, banerjee2021spectrum} -- and more \cite{chung1992laplacian, cvetkovic2005signless, mulas2021spectral, schaub2020random}. While running the gamut, many of these ``graph-generalization Laplacians" are borne of a similar desire to capture combinatorial structures inexpressible by graphs. \\

\paragraph{\bf Laplacian generalizations} Finally, ``Laplacian generalizations" cast multiple Laplacian variants as special cases of a more general definition. These more general definitions often arise naturally as more useful objects of study. For instance, the so-called generalized Laplacian, defined as any symmetric $n \times n$ matrix $Q$ satisfying
\[
\begin{cases}
Q_{i,j}<0 & \mbox{if $\{i,j\}$ is an edge} \\
Q_{i,j}=0 & \mbox{if $i\not=j$ are not adjacent}
\end{cases}
\]
is used to define the Colin de Verdi\'{e}re invariant of a graph \cite{vanderholst1995short}. This definition includes not only the combinatorial Laplacian, but also $-A$ as a Laplacian. In a similar vein, researchers in computer science sometimes find it advantageous to build computational algorithms under a more general definition; Cohen et al.~\cite{cohen2016faster} considers a Laplacian to be ``simply a matrix with non-positive off-diagonal entries such
that each diagonal entry is equal to the sum of the absolute value of the other off-diagonal entries
in that column."  Mathematically, this condition is equivalent to asserting that the all ones vector, $\one$, is in the left kernel of the matrix. Rather than relax these constraints on entries in Laplacians, other Laplacian generalizations aim to {\it parametrize} families of Laplacians, wherein specific Laplacians are recovered by particular parameter settings. The deformation Laplacian \cite{morbidi2013deformed} is a matrix polynomial,
\[
M(t)=I-At+(D-I)t^2,
\]
where $t$ is any complex number. Observe here that $M(1)$ recovers the graph Laplacian, $M(-1)$ the signless Laplacian, and $M(0)$ the identity.  In short, these and other Laplacian generalizations aim to unify Laplacian variants and thereby better understand how and why their differences arise. 

The prior work most germane to ours arguably belongs to all three of the aforementioned Laplacian categories. In \cite{Horak:WeightedLaplcian}, Horak and Jost propose a Hodge Laplacian inspired framework for weighted simplicial complexes, from which combinatorial and normalized Laplacians may be obtained. As we explain in detail in \Cref{S:IPL}, their framework crucially imposes an orthogonality assumption between faces of the simplicial complex. In this work, we follow in a similar vein to Horak and Jost, but drop this orthogonality requirement and consider Hodge Laplacians defined under general inner products; we call our more general framework the ``inner product Laplacian". Though seemingly subtle, this greater generality vastly enriches both the applicability of the inner product Laplacian in representing complex data and systems, as well as its theoretical underpinnings. With regard to the latter, developing a cogent spectral theory for the inner product Laplacian requires a careful consideration of the the relationship between inner product spaces and an orthonormal basis for that inner product space, which we flesh out next in \Cref{S:conformal}.

\section{Conformal Inner Product Spaces}\label{S:conformal}

A key technical tool we use to understand the properties of the inner product Laplacian is a new measure we call {\it conformality} that quantifies the relationship between an inner product space and a specified basis for the same inner product space.  Philosophically, the motivation for conformality is the practical observation that while an inner product space does not have a preferred representation or basis, oftentimes in applications there is a ontologically natural basis to associate with the inner product space.  As we will see in \Cref{S:IPL}, in the case of the inner product Laplacian for graphs the ontologically natural orthonomal basis is given by the vertices and edges of the graph.  As this concept may be of independent interest, we provide a self-contained discussion of conformality and some of its key properties.

\begin{define}[Conformality]
Let $\mathcal{E} = \set{e_1,\ldots, e_k}$ be an orthonormal basis for a $k$-dimensional real inner product space with $k\geq 2$.  Let $\mathcal{Z}$ be another $k$-dimensional real inner product space defined on the formal linear combinations of $\mathcal{E}$.
For $x \in \R^k$, let $\hat{x}$ denote the associated formal linear combination of elements of $\mathcal{E}$, i.e.\ $\hat{x} = \sum_i x_i e_i.$
 We say $\Z$ is \emph{strongly (weakly) $\rho$-conformal with respect to $\mathcal{E}$} if for all $x,y \in \R^k$ such that $\sum_i x_iy_i = 0$ (such that $x_iy_i = 0$ for all $i$), we have that
\[ \inner[\Z]{\hat{x},\hat{y}} \leq \rho \norm[\Z]{\hat{x}}\norm[\Z]{\hat{y}}.\] 

The \emph{strong (weak) conformality of $\Z$ with respect to $\mathcal{E}$} is the least $\rho$ such that $\Z$ is strongly (weakly) $\rho$-conformal.
\end{define}

Recall that given a finite-dimensional real inner product space $\Z$ and a basis for the inner product space $\mathcal{E}$, the inner product can be entirely specified by the positive definite matrix $M$ where the $M_{ij}$ entry is given by $\inner[\Z]{e_i,e_j}$.  Specifically, if $x, y \in \R^k$ and we define $\hat{x} = \sum_i x_i e_i$ and $\hat{y} = \sum_i y_i e_i$, then $\inner[\Z]{\hat{x},\hat{y}} = x^TMy$.  As we are mainly interested in the case where there is a natural basis for $\Z$, we will often abuse notation and refer to the $\Z$ by the matrix $M$ associated with the natural basis. In this case, we will say that $M$ \emph{represents} the inner product space $\Z$. Accordingly, we say that a positive definite matrix $M$ is strongly (weakly) $\rho$-conformal if $x^TMy \leq \rho \sqrt{x^TMx y^TMy}$ for all $x, y$ such that $x^Ty = 0$ ($x$ and $y$ have disjoint support).  As we will see in \Cref{L:similar}, the strong conformality is independent of the choice of basis (and hence the choice of representing matrix $M$), while weak conformality is closely tied to the form of the representing matrix $M$.  Intuitively,  this may be understood as a result of the conditions placed on $x$ and $y$ in the definition.  In the case of strong conformality the condition ``lives" in the underlying inner product space, while for weak conformality the independent support is inherently a combinatorial constraint dependent on the chosen basis, and hence, the matrix $M$. 

As a side note, the framing in terms of a representative positive definite matrix $M$ makes clear the existence of strong and weak conformality and that the weak conformality is bounded above by the strong conformality.  In particular, finding strong (weak) conformality is equivalent to finding the maximum of $x^TMy$ subject to the constraint that $x^TMx = y^TMy = 1$ and $x^Ty = 0$ ($x_iy_i = 0$ for all $i$).  Since $M$ is a positive definite matrix, the set of $(x,y)$ satisfying these constraints is closed and bounded (that is, compact) and thus there is a pair $(x^*,y^*)$ which achieves the maximum.  The observation that the strong conformality is an upper bound on the weak conformality follows immediately by from noting that if $x_iy_i =0$ for all $i$, then $x^Ty = 0$.

We first address the issue of explicitly calculating the conformality of a given positive definite matrix $M$.  Here, it is helpful to understand which matrix operations preserve the conformality of a matrix.  To that end, we make the following observation:

\begin{lemma}\label{L:similar}
If $M$ is the matrix representing a strongly $\rho$-conformal inner product space, then for any unitary matrix $U$, the matrix $U^T M U$ represents a strongly $\rho$-conformal inner product space. If $M$ is weakly $\rho$-conformal and $D$ is an invertible diagonal matrix, then $DMD$ represents a weakly $\rho$-conformal inner product space.
\end{lemma}

\begin{proof}
We will denote by $M^*$ the matrix representing the modified inner product space, that is, $M^* = U^TMU$ in the strong case and $M^* = DMD$ in the weak case.  Note that since $M$ can be written in the form $Q^TQ$, in both cases there is a $Q^*$ such that $M^* = (Q^*)^TQ^*$ and thus $M^*$ is positive semi-definite.  Furthermore, since the kernel of $M$, $U$, and $D$ are all trivial by definition, $M^*$ is positive definite and hence defines an inner product space. 

In order to prove the conformality of $M^*$, we note that it suffices to provide an invertible function $v \mapsto v^*$ such that for any $x, y$ we have $x^TMx = (x^*)^TM^*x^*$,  $y^TMy = (y^*)^TM^*y^*$, and  $y^TMx = (y^*)^TM^*x^*$, while in addition preserving the appropriate orthogonality condition: $\sum x_iy_i = 0$ if and only if $\sum x_i^*y_i^*=0$ or $x_iy_i = 0$ if and only if $x_i^*y_i^* = 0.$  In the case of strong conformality this function is defined by $v^* = Uv$ and for weak conformality this function is defined by $v^* = D^{-1}v.$  It is a straightforward exercise to confirm that the desired conditions hold for these functions.
\end{proof}

We note that \Cref{L:similar} implies that the strong conformality of an inner product space is \emph{basis-independent} and a property of the underlying inner product space.  In contrast, it is easy to see that weak conformality is a property of the specific representation of the inner product space.  In particular, we observe that the weak conformality of a diagonal matrix is 0, while any matrix with a non-zero off-diagonal entry has a strictly positive weak conformality. Thus, conjugating a matrix $M$ by a unitary matrix can, and often does, change the weak conformality of the matrix.  Consequently, it is unsurprising that, as a linear algebraic property, the strong conformality is relatively straightforward to calculate, while, as a more combinatorial property, the weak conformality is more challenging.  

\begin{corollary}\label{C:strong}
Let $M$ be a $k \times k$ positive definite matrix and let $\lambda_1 \geq \lambda_2 \geq \cdots \geq \lambda_k > 0$ be the eigenvalues.   The strong conformality of $M$ is given by \[\frac{\lambda_1 - \lambda_k}{\lambda_1 + \lambda_k}.\] 
\end{corollary}

\begin{proof}
By \Cref{L:similar}, we may assume without loss of generality that $M$ is a diagonal matrix with diagonals given by $\lambda_1 \geq \cdots \geq \lambda_k$.  In order to find the maximum conformality we first fix a sequence $\alpha_1,\ldots,\alpha_k$ such that $\sum_i \alpha_i = 0$ and restrict our attention to pairs $x,y \in \R^k$ such that $x_iy_i = \alpha_i$.  By construction, such pairs satisfy that $x^Ty = 0$ and $x^TMy = \sum_i \alpha_i \lambda_i.$  Thus, in order to find the maximum conformality among such pairs, it suffices to minimize $x^TMx y^TMy$ subject to the condition that $x_iy_i = \alpha_i$ for all $i$.  Because $M$ is assumed to be diagonal, the sign of a particular entry in $x$ or $y$ is irrelevant to the value.  Thus, we may reparameterize $x_i = \sqrt{\abs{\alpha_i}} p_i$ and $y_i = \sgn(\alpha_i) \sqrt{\abs{\alpha_i}} \frac{1}{p_i}$ for any strictly positive $p_i$. As a result, we have that
\begin{align*}
    x^TMx y^T M y &= \paren{\sum_i \lambda_i \abs{\alpha_i} p_i^2}\paren{\sum_j \lambda_j \abs{\alpha_j} \frac{1}{p_j^2}} \\
    &= \sum_i \lambda_i^2 \abs{\alpha_i}^2 + \sum_{i < j} \lambda_i \lambda_j \abs{\alpha_i}\abs{\alpha_j} \paren{ \frac{p_i^2}{p_j^2} + \frac{p_j^2}{p_i^2}} \\
\end{align*}
We note that since $t^2 + \frac{1}{t^2}$ is minimized when $t = \pm 1$ and all the $p_i$'s are positive, the minimum occurs when the $p_i$'s are identically $c$ for some positive $c$.  As a consequence, given the sequence $\alpha_1, \ldots, \alpha_k$ the associated maximal conformality is given by
\begin{equation}\frac{\sum_i \alpha_i \lambda_i}{\sum_i \abs{\alpha_i}\lambda_i}.  \tag{$*$} \label{E:conform}\end{equation}

At this point it suffices to find the sequence $\alpha_1, \ldots, \alpha_k$ which maximizes \eqref{E:conform}.  As the condition $\sum_i \alpha_i = 0$ is invariant under a uniform change of sign, we may assume without loss of generality that the $\alpha_i$'s are chosen so that $\sum_i \alpha_i \lambda_i > 0.$ Now note that for any $s, t > 0$ and $a \geq b > 0$, $\frac{s-a}{t + a} \leq \frac{s-b}{t+b}$.
Thus if we have that $\alpha_j < 0$ for some $j \in [k-1]$, then $\sum_{i \neq j} \alpha_i \lambda_i > 0$ and 
\[ \frac{\sum_{i \neq j} \alpha_i \lambda_i + \alpha_j\lambda_k}{\sum_{i \neq j} \abs{\alpha_i}\lambda_i + \abs{\alpha_j}\lambda_k} \geq  \frac{\sum_{i \neq j} \alpha_i \lambda_i + \alpha_j\lambda_j}{\sum_{i \neq j} \abs{\alpha_i}\lambda_i + \abs{\alpha_j}\lambda_j} \] as $\abs{\alpha_j\lambda_j} \geq \abs{\alpha_j\lambda_k}.$
Thus defining $\alpha'$ by $(\alpha_1, \cdots, \alpha_{j-1}, 0, \alpha_{j+1}, \cdots , \alpha_{k-1}, \alpha_k + \alpha_j)$, we have the value of $\eqref{E:conform}$ for $\alpha'$ is at least as large as that of $\alpha$.  Thus, we may assume without loss of generality that $-\alpha_k = \alpha_1 + \ldots + \alpha_{k-1}$ and $\alpha_1,\ldots,\alpha_{k-1} \geq 0.$  Indeed, as $\eqref{E:conform}$ is scale invariant we may assume that $\alpha_k = -1$ and $\alpha_1 + \cdots + \alpha_{k-1} = 1$.  But, as $\frac{z -  \lambda_n}{z+\lambda_n}$ is monotonically increasing in $z$, thus the maximum value is achieved when $\alpha_1 = 1$, $\alpha_2 = \cdots = \alpha_{k-1} = 0$, and $\alpha_k = -1$.
\end{proof}

A direct consequence of \Cref{C:strong} is that the strong conformality of $M$ is closely related to its condition number, $\kappa(M)$.  Indeed, the strong conformality may be rewritten as 
\[ \frac{\kappa(M) - 1}{\kappa(M) + 1} = 1 - \frac{2}{\kappa(M)+1}. \]
Thus if the strong conformality of $M$ is close to 1, then $M$ is poorly-conditioned, whereas if the conformality of $M$ is close to 0, then $M$ is well-conditioned.

Given the combinatorial nature of the disjoint support constraint it is natural to suspect that determining whether the weak conformality of a matrix $M$ exceeds a given $\rho$ would be \NPc.  Indeed,  the superficial similarity of the weak conformality to known \NPc\ problems such as \texttt{MaxCut} and \texttt{nCut}, suggests a broad class of candidate problems to reduce to the decision problem for weak conformality.  In fact, in \Cref{L:NPc} we will provide a surprisingly elegant reduction from the \texttt{Partition} problem.\footnote{The partition problem asks, given a set $x_1,\ldots, x_n$ of natural numbers, does there exist a partition $(S,\overline{S})$ so that $\sum_{s \in S} x_s = \sum_{t \in \overline{S}} x_t$.}  However, before proceeding to \Cref{L:NPc} it will be helpful to better understand the computational road-blocks to computing weak conformality.  In particular, we first consider the easier case of maximizing over a fixed support for $x$ and $y$.  To that end, fix a positive definite matrix $M$ and a set of indices $S$.  We wish to find the maximum value of $\frac{ x^TMy}{\sqrt{x^TMx y^TMy}}$ conditioned on $S$ containing the support of $x$ and $\overline{S}$ containing the support of $y$.  In order to simplify notation, we denote by $M_{AB}$  the submatrix of $M$ formed by rows in $A$ and columns in $B$.  Thus, the maximization problem to reduces to finding $x \in \R^{\abs{S}}$ and $y \in \R^{\abs{\overline{S}}}$, which maximizes 
\[ \frac{x^TM_{S\overline{S}}y}{\sqrt{x^TM_{SS}x y^TM_{\overline{S}\overline{S}}y}}.\] 
Since $M_{SS}$ and $M_{\overline{S}\overline{S}}$ are principle submatrices of $M$, they are positive definite and hence there exists invertible matrices $Q_S$ and $Q_{\overline{S}}$  such that $M_{SS} = Q_S^TQ_S$ and $M_{\overline{S}\overline{S}} = Q_{\overline{S}}^TQ_{\overline{S}}$. Defining $\widehat{x} = Q_Sx$ and $\widehat{y} = Q_{\overline{S}}y$, the maximization problem further reduces to $\hat{x}^T Q_S^{-T}M_{S\overline{S}} Q_{\overline{S}}^{-1} \hat{y}$ where $\norm[1]{\hat{x}} = \norm[1]{\hat{y}} = 1$.  It follows immediately that the maximum value is given by the largest singular value of 
$Q_S^{-T}M_{S\overline{S}} Q_{\overline{S}}^{-1}$.  Alternatively, the maximum is 
\[\norm[2]{Q_{S}^{-T}M_{S\overline{S}} Q_{\overline{S}}^{-1}Q_{\overline{S}}^{-T} M_{S\overline{S}}^TQ_{S}^{-1}}^{\nicefrac{1}{2}} = \norm[2]{Q_{S}^{-T}M_{S\overline{S}} M_{\overline{S}\overline{S}}^{-1} M_{S\overline{S}}^TQ_{S}^{-1}}^{\nicefrac{1}{2}}  = \norm[2]{M_{SS}^{-1}M_{S\overline{S}} M_{\overline{S}\overline{S}}^{-1} M_{S\overline{S}}^T}^{\nicefrac{1}{2}}, \]
where the last equality comes from standard symmetrization arguments for the spectrum.
Finally, we note that the explicit calculation of the outer matrix inverse can be avoided by instead solving a generalized eigenvalue problem $\lambda M_{SS} v = M_{S\overline{S}} M_{\overline{S}\overline{S}}^{-1} M_{S\overline{S}}^T v$.  Thus, together with the symmetry of $x$ and $y$, the weak conformality of $M$ may be calculated by finding the maximum eigenvalue in a collection of $2^{k-1}-1$ different generalized eigenvalue problems.  Now we are prepared to turn to the complexity of the decision problem for weak conformality.

\begin{lemma}\label{L:NPc}
    The decision problem for weak conformality is \NPc.
\end{lemma}

\begin{proof}
    We note that since matrix inverse and spectral radius of a matrix can be computed in polynomial time, the reformulation of weak conformality in terms of an exponential family of such problems immediately yields membership in \NP. It thus remains to show that the \texttt{Partition} problem can be reduced to the weak conformality decision problem.  In particular, given an instance of \texttt{Partition}  we will show how to construct a positive definite matrix $M$ and a value $\rho$ such that the instance has an affirmative answer if and only if the weak conformality of $M$ is at least $\rho$.

    To that end, let $x_1,\ldots, x_n$ be a sequence of natural numbers and let $2X = \sum_i x_i$.  Define the vector $x$ as the vector of square roots $(\sqrt{x_1}, \ldots, \sqrt{x_n})$ and let $M = xx^T + I$.  Extending the notation from above, for any subset of the indices $S$ we denote by $x_S$ the restriction of $x$ to the entries in $S$.  We will further define $\tilde{x}_S$ to be the unit vector in the direction of $x_S$.  Now for any fixed set $S$, we have that $M_{SS} = x_Sx_S^T + I = \norm[2]{x_s}^2 \tilde{x}_S\tilde{x}_S^T +  I$, $M_{S\overline{S}} = x_Sx_{\overline{S}}^T$, and $M_{SS}^{-1} = \frac{1}{\norm[2]{x_S}^2 + 1} \tilde{x}_S\tilde{x}_S + P_S$ where $P_S$ is the projection operator onto the space orthogonal to $x_S$.  Thus we have that 
    \begin{align*} M_{SS}^{-1}M_{S\overline{S}}M_{\overline{S}\overline{S}}^{-1}M_{S\overline{S}}^T &= \paren{\frac{1}{\norm[2]{x_S}^2 +1}\tilde{x}_S\tilde{x}_S^T + P_S}x_Sx_{\overline{S}}^T\paren{\frac{1}{\norm[2]{x_{\overline{S}}}^2+1}\tilde{x}_{\overline{S}}\tilde{x}_{\overline{S}}^T + P_{\overline{S}}}x_{\overline{S}}x_S^T \\
    &= \paren{ \frac{\norm[2]{x_S}}{\norm[2]{x_S}^2 +1}\tilde{x}_Sx_{\overline{S}}^T}\paren{\frac{\norm[2]{x_{\overline{S}}}}{\norm[2]{x_{\overline{S}}}^2+1}\tilde{x}_{\overline{S}}x_{S}^T} \\
    &= \frac{\norm[2]{x_S}^2\norm[2]{x_{\overline{S}}}^2}{\paren{\norm[2]{x_S}^2 +1}\paren{\norm[2]{x_{\overline{S}}}^2 +1}}\tilde{x}_Sx_S^T.
    \end{align*}
    Noting that $\norm[2]{x_S}^2 = \sum_{s \in S} x_s$, we have that  
    \[ \norm[2]{M_{SS}^{-1}M_{S\overline{S}}M_{\overline{S}\overline{S}}^{-1}M_{S\overline{S}}^T}^{\half} = \sqrt{ \frac{(X-t)(X+t)}{(X-t+1)(X+t+1)} } = \sqrt{\frac{X^2 - t^2}{(X+1)^2 - t^2}}, \]
    where $t$ is $\abs{ \sum_{s \in S} x_s - \sum_{t \in \overline{S}} x_t}.$  As $\frac{X^2-t^2}{(X+1)^2 - t^2}$ is a decreasing function of $t$ for positive $t$, this implies that if $x_1,\ldots,x_n$ is an affirmative instance, then the weak conformality of $M$ is $\frac{X}{X+1}$, while if it is negative instance it is at most 
    \[ \sqrt{ \frac{X^2 - 1}{(X+1)^2 - 1}}.\]  Thus a solution to the weak conformality decision problem for any $\rho$ in the interval between these two values provides a solution to the specified \texttt{Partition} instance, showing the \NPc ness of the decision problem for weak conformality. 
\end{proof}

It is worth mentioning that the above discussion implicitly assumes a computational model where infinite precision arithmetic can be done in constant time.  As conformality is is invariant under re-scaling (as noted in \Cref{L:similar}) and the weak conformality gap between the positive and negative instances for the \texttt{Partition} problem depends polynomially on the size of the instance, in principle it should be possible to extend the above result to the reader's preferred computational model of finite precision arithmetic. 

Given the computational challenges involved in calculating the weak conformality, it is natural to ask whether the strong conformality serves as an easily computable approximation for the weak conformality.  However, the next family of examples shows that the strong conformality can be an arbitrarily poor approximation for the weak conformality.

\begin{lemma}
Let $0 \leq \rho_w \leq \rho_s < 1$.  For any $k \geq 2$, there exists a positive definite matrix $M \in \R^{k \times k}$ such that the weak conformality of $M$ is $\rho_w$ and the strong conformality of $M$ is $\rho_s$.
\end{lemma}

\begin{proof}
We first consider the case where $k = 2$ and let the matrix \[M = \begin{bmatrix} \alpha & \rho_w \\ \rho_w & \frac{1}{\alpha} \end{bmatrix}\] with $\alpha \geq 1 \geq \frac{1}{\alpha}$. 
As $\det(M) = 1 - \rho_w^2 \in (0,1]$ and the entries of $M$ are non-negative, $M$ is a positive definite.  To determine the weak conformality of $M$, we note that there is precisely one (unordered) pair $x$ and $y$ such that $x^TMx = y^TMy = 1$ and $x$ and $y$ have disjoint support, namely $\set{ (\frac{1}{\sqrt{\alpha}}, 0), (0,\sqrt{\alpha})}$.  This pair immediately gives that the weak conformality of $M$ is $\rho_w$.  

We now consider the strong conformality of $M$.  By \Cref{C:strong} this can written in terms of the eigenvalues of $M$, which are 
\[ \frac{ \alpha + \frac{1}{\alpha} \pm \sqrt{ (\alpha - \frac{1}{\alpha})^2 + 4\rho_w^2}}{2}.\]
Thus the strong conformality is 
\[ \frac{\sqrt{ (\alpha - \frac{1}{\alpha})^2 + 4\rho_w^2}}{\alpha + \frac{1}{\alpha}}.\] 
As $\alpha$ ranges over $[1,\infty)$, the strong conformality of $M$ ranges over $[\rho_w,1)$, and thus there exists some choice of $\alpha$ such that $M$ has strong conformality $\rho_s.$

To extend this result to $k > 2$, we adjoin an appropriately sized identity matrix and consider the matrix 
\[ M_k = \begin{bmatrix} M & 0 \\ 0 & I \end{bmatrix}.\]
As the adjoined matrix is diagonal, this weak conformality of $M_k$ is the same as $M$. Further, for any choice of $\alpha$
\[ \frac{ \alpha + \frac{1}{\alpha} - \sqrt{ (\alpha - \frac{1}{\alpha})^2 + 4\rho_w^2}}{2} \leq 1 \leq \frac{ \alpha + \frac{1}{\alpha} + \sqrt{ (\alpha - \frac{1}{\alpha})^2 + 4\rho_w^2}}{2}\]
and so the minimum and maximum eigenvalues of $M_k$ are the same as $M$, resulting in the same strong conformality.
\end{proof}

The following two lemmas will be essential to our proofs in \Cref{S:iso} of generalizations of the Cheeger inequality and the expander mixing lemma to the case of the inner product Laplacians.  In essence, these two lemmas can be thought of as a means to allow for an arbitrary positive definite matrix to be effectively approximated by diagonal matrix.  While both of the lemmas are phrased in terms of weak conformality, they immediately extend to strong conformality as the strong conformality is an upper bound on the weak conformality.

\begin{lemma}\label{L:conformal}
Let $M \in \R^{k\times k}$ be a $\rho$-weakly conformal matrix.  Let $x \in \R^k$ and define $\abs{x}$ as the component-wise absolute value of $x$, then 
\[ \frac{1-\rho}{1+\rho} \abs{x}^TM\abs{x} \leq x^TMx \leq \frac{1+\rho}{1-\rho}\abs{x}^TM\abs{x}.\] 
\end{lemma}
\begin{proof}
Let $x_+$ and $x_-$ be vectors in $\R^k_+$ with disjoint support such that $x_+ - x_- = x$.  Define $a = x_+^TMx_+$, $b = x_-^TMx_-$ and $c = x_+^TMx_-$. Since $\abs{x} = x_+ + x_-$, we have that 
\[ \frac{x^TMx}{\abs{x}^TM\abs{x}} = \frac{a - 2c + b}{a+2c+b} = 1 - \frac{4c}{a + 2c + b}.\]
Taking the derivative with respect to $c$, it is easy to see that this is monotonically decreasing with respect to $c$, thus the maximum value is achieved when $c$ is minimized and the minimum value is achieved when $c$ is maximized.  Since $x_+$ and $x_-$ have disjoint support and $M$ is weakly $\rho$-conformal, by applying the arithmetic-geometric mean inequality  we have that \[ -\rho\frac{a+b}{2} \leq -\rho\sqrt{ab} \leq c \leq \rho\sqrt{ab} \leq \rho \frac{a+b}{2}.\]
Choosing the extremal values of $c$, we have that 
\[ \frac{1-\rho}{1+\rho} = 1  - \frac{2\rho}{1+\rho} \leq \frac{x^TMx}{\abs{x}^TM\abs{x}} \leq 1 + \frac{2\rho}{1-\rho} = \frac{1+\rho}{1-\rho},\] 
which completes the proof. 
\end{proof}

While \Cref{L:conformal} enables using the conformality of a matrix to adjust the signs of the off diagonal entries of the associated quadratic form, \Cref{L:trace} allows for the direct approximation of the quadratic form by the diagonal terms of the quadratic form.

\begin{lemma}\label{L:trace}
Let $M \in \R^{k \times k}$ be a positive definite matrix representing a $\rho$-weakly conformal inner product space.  Then, 
\[ \frac{1-\rho}{1+\rho} \trace(M) \leq \one^TM\one \leq \frac{1+\rho}{1-\rho}\trace(M).\]
Further, for any vector $x \in \R^k$, 
\[ \frac{1-\rho}{1+\rho} \sum_i x_i^2 M_{ii} \leq x^TMx \leq \frac{1+\rho}{1-\rho}\sum_i x_i^2M_{ii}.\]
\end{lemma}

\begin{proof}
We first note that for any $2 \leq j \leq k-1$ and any matrix $X \in \R^{k \times k}$, $\sum_{S \in \binom{[k]}{j}} \one^T_S X \one_S$ can be written in terms of $\one^TX\one$ and $\trace(X)$.  In particular, any off-diagonal term $x_{ij}$ appears precisely $\binom{k-2}{j-2}$ times in the sum, while the diagonal terms appear $\binom{k-1}{j-1}$ times.  Thus, by rearranging, we have 
\[ \sum_{S \in \binom{[k]}{j}} \one_S^T X \one_S = \binom{k-2}{j-2} \one^TX\one + \paren{\binom{k-1}{j-1} - \binom{k-2}{j-2}} \trace(X) = \binom{k-2}{j-2} \one^TX\one + \binom{k-2}{j-1} \trace(X).\]
  Note that by making the standard abuse of notation that the number of ways to choose a negative number of elements is 0, this can be extended to the case where $j=1.$

Thus, for the weakly $\rho$-conformal matrix $M$, we have that for  any $1 \leq j \leq k-1$
\begin{align*}
\binom{k}{j} \one^TM\one &= \sum_{S \in \binom{[k]}{j}} \paren{\one_S + \one_{\overline{S}}}^T M \paren{\one_S + \one_{\overline{S}}} \\
&= \sum_{S \in \binom{[k]}{j}} \one_S^TM\one_S + 2\one_S^TM\one_{\overline{S}} + \one_{\overline{S}}^TM\one_{\overline{S}} \\
&\geq \sum_{S \in \binom{[k]}{j}} \one_S^TM\one_S - 2\rho\sqrt{\one_S^TM\one_S\one_{\overline{S}}^TM\one_{\overline{S}}} + \one_{\overline{S}}^TM\one_{\overline{S}} \\
&= \sum_{S \in \binom{[k]}{j}} \paren{ \one_S^TM\one_S  + \one_{\overline{S}}^TM\one_{\overline{S}}} - 2\rho \sum_{S \in \binom{[k]}{j}} \sqrt{\one_S^TM\one_S\one_{\overline{S}}^TM\one_{\overline{S}}} \\
&\geq (1-\rho) \sum_{S \in \binom{[k]}{j}} \paren{ \one_S^TM\one_S  + \one_{\overline{S}}^TM\one_{\overline{S}}} \\
&= (1-\rho) \paren{\binom{k-2}{j-2} + \binom{k-2}{k-j-2}}\one^TM\one + (1-\rho) \paren{\binom{k-2}{j-1} + \binom{k-2}{k-j-1}}\trace(M) \\
&= (1-\rho) \paren{\binom{k-2}{j-2} + \binom{k-2}{j}}\one^TM\one + 2(1-\rho) \binom{k-2}{j-1}\trace(M), \\
\end{align*}
where the first inequality follows from the weak $\rho$-conformality of $M$ and the second follows from the arithmetic-geometric mean inequality.  By solving for $\one^TM\one$ we have that 
\[ \one^TM\one \geq \frac{2(1-\rho)\frac{j(k-j)}{k(k-1)}}{1 - (1-\rho)\paren{\frac{j(j-1) + (k-j)(k-j-1)}{k(k-1)}}}\trace(M) = \frac{2(1-\rho) \frac{j(k-j)}{k(k-1)}}{\rho +2(1-\rho)\frac{j(k-j)}{k(k-1)}} \trace(M).\]
We note that the same argument, with appropriate sign changes, will also upper bound $\one^TM\one$ in terms of $\trace(M)$ and in particular, we have that for any $1 \leq j \leq k-1$,
\[ \paren{1 - \frac{\rho}{2(1-\rho)\frac{j(k-j)}{k(k-1)} + \rho}}\trace(M) \leq \one^TM\one \leq \paren{1 + \frac{\rho}{2(1+\rho)\frac{j(k-j)}{k(k-1)} - \rho}}\trace(M).\]
Noting that $\max_{1 \leq j \leq k-1} \frac{j(k-j)}{k(k-1)} > \frac{1}{4}$ gives the desired result.

In order to recover the result for arbitrary $x \in \R^k$, we first consider the case where no entry of $x$ is zero.  Recall from \Cref{L:similar}, that if $M$ is weakly $\rho$-conformal then $DMD$ is also weakly $\rho$-conformal for any invertible diagonal matrix $D$. If we define $D_x$ as the diagonal matrix with $x$ along the diagonal, then $D_x$ is invertible, $x^TMx = \one^TD_xMD_x\one$,  and $D_xMD_x$ is $\rho$-conformal.  As $\trace(D_xMD_x) = \sum_i x_i^2 M_{ii}$, we have the desired result.  To handle the case where $x$ has zero entries, we note that if $M$ is $\rho$-conformal (strongly or weakly) then any (non-trivial) principle minor of $M$ is also $\rho$-conformal.  This can be seen by adding conditions of the form $x_i = y_i = 0$ for the non-principle indices to the optimization formulations of conformality.
\end{proof}

It is worth noting that both the upper and lower bound in \Cref{L:trace} are asymptotically tight.  To see this, consider the $k\times k$ matrix $M = I + \frac{\alpha}{k}J$.  Clearly for $\alpha > -1$, this matrix is positive definite, and further $\one^TM \one = k + \alpha k = (1+\alpha)k$ and $\trace(M) = k + \alpha$. Now to determine the weak conformality of $M$, we note that by re-scaling we may restrict our attention to $x$ and $y$ such that $\abs{\one^Tx} = \abs{\one^Ty} = k$.  In this case, the conformality ratio is given by 
\[ \frac{k\abs{\alpha}}{\sqrt{\paren{\norm[2]{x}^2 + k\alpha}\paren{\norm[2]{y}^2 + k\alpha}}}.\]
By standard results, if $x$ or $y$ has support of size $t$, then the minimum of $\norm[2]{x}^2$ given that $\abs{\one^Tx} = k$ is $\nicefrac{k^2}{t}$.  It is then easy to see that if $k$ is even, the conformality ratio is maximized when $t = \nicefrac{k}{2}$ and the conformality of $M$ is $\nicefrac{\abs{\alpha}}{2+\alpha}$.  Note that for $\alpha \in (-1,0)$ the weak conformality takes on values in $(0,1)$ and for $\alpha \in (0,\infty)$ the weak conformality takes on values in $(0,1)$.  Hence, depending on the choice of $\alpha$ the matrix $M$ takes on all non-zero, non-one possible values of the weak conformality twice, once for positive $\alpha$ and once for negative $\alpha.$  

Combining these results and applying \Cref{L:trace} we have that \[ \frac{2 +\alpha - \abs{\alpha}}{2 + \alpha + \abs{\alpha}} (k+\alpha) \leq \one^TM\one \leq \frac{2 + \alpha + \abs{\alpha}}{2 + \alpha - \abs{\alpha}} (k+\alpha).\]
If $\alpha \leq 0$, then the lower bound simplifies to $(1+\alpha)k+(1+\alpha)\alpha$ while if $\alpha \geq 0$ the upper bound simplifies to $(1+\alpha)k+(1+\alpha)\alpha$.  As $\one^TM\one = (1+\alpha)k$, we have that as $k \rightarrow \infty$ the lower bound is asymptotically tight if $\alpha \leq 0$ and the upper bound is asymptotically tight if $\alpha \geq 0$.  Thus, for any fixed $\rho$, there are two sequences of weakly $\rho$-conformal matrices such that one witnesses the tightness of the lower bound in \Cref{L:trace} and one witnesses the asymptotic tightness of the upper bound in \Cref{L:trace}.

While not directly relevant to the applications of conformality to understanding the behavior of the inner product Laplacian, the following straightforward observation indicates that both strong and weak conformality are measuring essential features of the isomorphism between the vector space defined with the natural basis and its dual space defined by the inner product.

\begin{lemma}\label{L:inverse}
For any positive definite matrix $M$, the strong- and weak-conformalities of $M$ are equal to the strong- and weak-conformalities of $M^{-1}$.
\end{lemma}

\begin{proof}
Let $\lambda$ and $\sigma$ be the largest and smallest eigenvalues of $M$, respectively.  Then the largest and smallest eigenvalues of $M^{-1}$ are $\nicefrac{1}{\sigma}$ and $\frac{1}{\lambda}$.  By \Cref{C:strong}, the strong conformality of $M$ is $\nicefrac{\lambda - \sigma}{\lambda + \sigma}$, while the strong-conformality of $M^{-1}$ is
\[ \frac{\frac{1}{\sigma} - \frac{1}{\lambda}}{\frac{1}{\sigma} + \frac{1}{\lambda}} = \frac{\frac{\lambda-\sigma}{\lambda\sigma}}{\frac{\lambda + \sigma}{\lambda\sigma}} = \frac{\lambda - \sigma}{\lambda + \sigma},\] as desired.

In order to show that the weak conformality of $M$ and $M^{-1}$ are the same, we show that for a fixed $S$, the spectral norm of $M_{SS}^{-1}M_{S\overline{S}}M_{\overline{S}\overline{S}}^{-1}M_{S\overline{S}}^T$ is the same whether $M$ is considered or $M^{-1}$.  To that end, consider $M$ as block matrix 
\[ M = \begin{bmatrix} A & B \\ B^T & D\end{bmatrix} \]
representing the partition $(S,\overline{S})$.  With this notation,  $M_{SS}^{-1}M_{S\overline{S}}M_{\overline{S}\overline{S}}^{-1}M_{S\overline{S}}^T$ can be rewritten as $A^{-1}BD^{-1}B^T$.  Now, as $A$ and $D$ are invertible, by standard results 
\[ \widehat{M} = M^{-1} = \begin{bmatrix} \paren{A-BD^{-1}B^T}^{-1} & - \paren{A-BD^{-1}B^T}^{-1}BD^{-1} \\ - \paren{D-B^TA^{-1}B}^{-1}B^TA^{-1} & \paren{D-B^TA^{-1}B}^{-1} \end{bmatrix}.\]
Thus we have that $\widehat{M}_{SS}^{-1}\widehat{M}_{S\overline{S}}\widehat{M}_{\overline{S}\overline{S}}^{-1}\widehat{M}_{S\overline{S}}^T$ can be rewritten as 
\begin{align*}
    &\paren{A-BD^{-1}B^T}\paren{A-BD^{-1}B^T}^{-1}BD^{-1}\paren{D-B^TA^{-1}B}\paren{D-B^TA^{-1}B}^{-1}B^TA^{-1} \\ 
    &= BD^{-1}B^TA^{-1} \\
    &= \paren{A^{-1}BD^{-1}B^T}^T \\
    &= M_{SS}^{-1}M_{S\overline{S}}M_{\overline{S}\overline{S}}^{-1}M_{S\overline{S}}^T.
\end{align*} 
As the weak-conformality can be calculated in terms of the spectral norms of $M_{SS}^{-1}M_{S\overline{S}}M_{\overline{S}\overline{S}}^{-1}M_{S\overline{S}}^T$ as $(S,\overline{S})$ ranges over all partitions of the rows/columns, this gives that the weak-conformality of $M$ and $M^{-1}$ are the same.
\end{proof}

\section{The Inner Product (Hodge) Laplacian}\label{S:IPL}
As our approach is inspired by the work of Horak and Jost defining Hodge Laplacians for a weighted simplicial complex~\cite{Horak:WeightedLaplcian}, we briefly review the necessarily simplicial homology background and the Horak and Jost simplicial Laplacian.\footnote{We note that the discussion here does not exactly correspond with that in \cite{Horak:WeightedLaplcian} as their exposition is focused on coboundary maps rather than boundary maps.  However, the essential ideas and the Laplacians defined on each of the chain groups are the same.  We choose to work with boundary operators as in applications the boundary of the $i$-faces tends to be more directly interpretable. }
An \emph{abstract simiplicial complex} $K$ over a finite set $X$ is a subset of $2^X - \set{\varnothing}$ such that if $A \in K$ and $B \in 2^X - \set{\varnothing}$ such that $B \subset A$, then $B \in K$.  The elements of $K$ are called its \emph{faces} and the maximal faces with respect the inclusion ordering are called the \emph{facets} of the simplicial complex.  The dimension of any face $A \in K$, is $\size{A} - 1$ and the dimension of the simplicial complex is the maximum of the dimension of all the faces.  For any fixed $i$, we denote the set of  faces of dimension $i$, $i$-faces for short, by $S_i$.  The formal linear combinations of elements of $S_i$ is a vector space, which is referred to as the \emph{$i^{\textrm{th}}$ chain group} and denoted $C_i(K,\R).$  By fixing an arbitrary ordering of the ground set $X$, the dimension $i$ faces of $K$ can be uniquely identified with elements of $X^{i+1}$ which agree with the prescribed ordering.  We will denote this form of the elements of $K$ by $(v_0,\ldots,v_{i})$ and denote by $(v_0,\ldots,\hat{v}_j,\ldots,v_{i})$ the face of $K$ formed by removing the $j^{th}$ component of $(v_0,\ldots, v_{i})$.  With this notation in hand, we define \emph{$i^{\textrm{th}}$ (signed) boundary map}, $\boundary_i$ as the unique function (up to ordering of $X$) from $C_i(K,\R)$ to $C_{i-1}(K,\R)$ formed by linearly extending the mapping $(v_0,\ldots,v_i) \mapsto \sum_{j} (-1)^j (v_0,\ldots,\hat{v}_j,\ldots,v_{i})$.  It is easy to see that under these maps, the action of first removing $v_j$ and then $v_k$ has the opposite sign as $v_k$ and then $v_j$.  Thus, $\partial_{i-1}\partial_i$ is identically zero.  A similar statement holds for the co-boundary maps $\set{\boundary_i^*},$ where $\boundary^*$ is the adjoint of the linear operator $\boundary.$

The collection of chain groups and boundary maps form a {\it chain complex}, where composition of consecutive maps is the zero map, for the simplicial complex, 
\[ 0= C_{-1}(K,\R) \xleftarrow{\boundary_0} C_{0}(K,\R) \cdots C_{k-2}(K,\R) \xleftarrow{\boundary_{k-1}} C_{k-1}(K,\R) \xleftarrow{\boundary_{k}} C_{k}(K,\R).\]
 From this chain complex, the $i^{\textrm{th}}$ Hodge Laplacian for the simplicial complex can be defined as 
\[\mathcal{L}_i  = \boundary_{i}^*\boundary_{i} + \boundary_{i+1}\boundary_{i+1}^*. \]   Further, by specifying an ordering of each of the sets of $i$-faces, the boundary maps can be written as $B_i \in \set{-1,0,1}^{\size{S_{i-1}} \times \size{S_i}}$ and the Hodge Laplacian can be written as $\mathcal{L}_i = B_i^TB_i + B_{i+1}B_{i+1}^T.$\footnote{From this point on, we will ignore the need to specify on an ordering on the $i$-faces to define the appropriate matrix and assume that every collection of $i$-faces has such an ordering and that it is consistently applied wherever appropriate.}  In a similar fashion a weight function on $K$, $w \colon K \rightarrow \R^+$ can be written as a collection of diagonal matrices $W_i$ corresponding to each of the non-empty collection of faces of $K$.  Using this framework, Horak and Jost define the collection of Hodge Laplacians on the weighted simplicial complex as
\[ \mathcal{L}_i(w) = B_i^TW_{i-1}^{-1} B_i W_i + W_i^{-1} B_{i+1} W_{i+1} B_{i+1}^T.\]

It is worth mentioning that while the $\mathcal{L}_i(w)$'s are not (in general) Hermitian since the matrices are not symmetric, the associated matrix $W_{i}^{\nicefrac{1}{2}}\mathcal{L}_i(w)W_i^{-\nicefrac{1}{2}}$ is positive semi-definite and co-spectral with $\mathcal{L}_i(w)$, so the spectrum of $L_i(w)$ is real and non-negative. 

As Horak and Jost note in passing, the weight matrix $W_i$ may be thought of as defining an inner product space on the chain group $C_i(K,\R)$ where the faces of dimension $i$ are orthogonal, but not necessarily orthonormal.  We generalize their approach by considering an arbitrary inner product on the chain complex, and in the process, recover a symmetric alternative to their Laplacian for a weighted simplicial complex.  To this end, let $\Z_i(K,\R)$ be an arbitrary real inner product space\footnote{Throughout we will assume real inner product spaces, but we this can be trivially generalized to complex inner product spaces.} over the vector space $C_i(K,\R)$.  
Since $\Z_i$ is finite dimensional, there is some symmetric positive definite matrix $M_i$ (in the basis given by $i$-faces) such that the $\inner[\Z_i]{x,y} = x^TM_iy$. Since $M_i$ is positive definite there is an invertible and symmetric matrix $Q_i$ such that $M_i = Q_i^TQ_i = Q_iQ_i$, for instance if singular value decomposition of $M_i = U_i^T\Sigma_i^2U_i$, then $Q_i$ may be defined as $U_i^T\Sigma_iU_i.$  We note that $Q_i : \Z_i \rightarrow C_i(K,\R)$ is an isometry, that is, for any $x,y \in \Z_i$, we have that $\inner[\Z_i]{x,y} = \inner{Q_ix,Q_iy}.$ 

Now let $\boundary_i \colon C_i \rightarrow C_{i-1}$ be the standard simplicial boundary maps, and define for an inner product space $\Z_i$ the \emph{inner product boundary map} $\zeta_i : \Z_i \rightarrow \Z_{i-1}$ by $\zeta_i = Q_{i-1}^{-1} \boundary_i Q_{i}$.  Note that the inner product boundary maps are chosen so that the following diagram commutes:
\begin{center}
\begin{tikzcd}
C_{i-1}(K,\R)\arrow[d,xshift=0.5ex,left,"Q_{i-1}^{-1}"]
& C_i(K,\R) \arrow[l,"\boundary_i"']   \arrow[d,xshift=0.5ex,left,"Q_{i}^{-1}"]
&  C_{i+1}(K,\R)  \arrow[l,"\boundary_{i+1}"'] \arrow[d,xshift=0.5ex,left,"Q_{i+1}^{-1}"]   \\
\Z_{i-1}(K,\R) 
\arrow[u,xshift=-0.5ex,right,"Q_{i-1}"] 
& \Z_i(K,\R) \arrow[l,"\zeta_i"'] 
\arrow[u,xshift=-0.5ex,right,"Q_{i}"] 
&  \Z_{i+1}(K,\R) \arrow[l,"\zeta_{i+1}"'] \arrow[u,xshift=-0.5ex,right,"Q_{i+1}"]  
\end{tikzcd}.
\end{center}

Now we wish to define the \emph{inner product co-boundary maps}, $\set{\zeta_i^*}$. Paralleling the standard theory of simplicial homology, we view the $\zeta^*$ as the adjoint of $\zeta$ with respect to the appropriate inner product spaces.  Thus if the $B_i$ is the matrix representation of $\boundary_i$ (and hence $B_i^T$ is the matrix representation of the co-boundary map $\boundary_i^*$), then $\zeta_i$ has matrix representation $Q_{i-1}^{-1}B_iQ_i$.  Thus, the co-boundary map associated with $\zeta_i$, denoted $\zeta_i^*$ must satisfy that 
\[ \inner[\Z_{i-1}]{\zeta_ix,y} = \inner[\Z_i]{x,\zeta_i^*y} \]
for all $x \in \Z_i$, $y \in \Z_{i-1}$.  Abusing notation and using $\zeta_i$ and $\zeta_i^*$ interchangeably with their matrix notation, this gives that 
\begin{align*}
    x^TQ_i^TQ_i\zeta_i^*y &= \inner[\Z_i]{x,\zeta_i^*y} \\
    &= \inner[\Z_{i-1}]{\zeta_ix,y} \\
    &= x^T\zeta_i^TQ_{i-1}^TQ_{i-1}y \\
    &= x^T\paren{Q_{i-1}^{-1}B_iQ_i}^TQ_{i-1}^TQ_{i-1}y \\
    &= x^TQ_{i}^{T}B_i^TQ_{i-1}^{-T}Q_{i-1}^TQ_{i-1}y \\
    &= x^TQ_{i}^{T}B_i^TQ_{i-1}y.
\end{align*}
Thus, since $Q_i$ is invertible and this relation holds for all $x \in \Z_i$, $y \in \Z_{i-1}$ the matrix form of $\zeta_i^*$ is given by $Q_i^{-1}B_i^TQ_{i-1}$ and $\zeta_i^* = Q_i^{-1} \boundary_i^* Q_{i-1}$.  As a result, the following diagram commutes:
\begin{center}
\begin{tikzcd}
C_{i-1}(K,\R) \arrow[r,yshift=-0.5ex,"\boundary_i^*"']\arrow[d,xshift=0.5ex,left,"Q_{i-1}^{-1}"]
& C_i(K,\R) \arrow[l,yshift=0.5ex,"\boundary_i"']  \arrow[r,yshift=-0.5ex,"\boundary_{i+1}^*"'] \arrow[d,xshift=0.5ex,left,"Q_{i}^{-1}"]
&  C_{i+1}(K,\R)  \arrow[l,yshift=0.5ex,"\boundary_{i+1}"'] \arrow[d,xshift=0.5ex,left,"Q_{i+1}^{-1}"]   \\
\Z_{i-1}(K,\R) \arrow[r,yshift=-0.5ex,"\zeta_i^*"'] 
\arrow[u,xshift=-0.5ex,right,"Q_{i-1}"] 
& \Z_i(K,\R) \arrow[l,yshift=0.5ex,"\zeta_i"'] \arrow[r,yshift=-0.5ex,"\zeta_{i+1}^*"']
\arrow[u,xshift=-0.5ex,right,"Q_{i}"] 
&  \Z_{i+1}(K,\R) \arrow[l,yshift=0.5ex,"\zeta_{i+1}"'] \arrow[u,xshift=-0.5ex,right,"Q_{i+1}"]  \\
\end{tikzcd}
\end{center}

It is worth mentioning at this point that the inner product boundary and co-boundary operators satisfy many of the anticipated properties of a boundary or co-boundary operator. For example; 
\begin{lemma}\label{L:doubleB}
Let $\set{\zeta_i}$ and $\set{\zeta_{i}^*}$ be a collection of inner product boundary and co-boundary maps.  Then $\zeta_{i-1}  \zeta_i = 0 = \zeta_i^*  \zeta_{i+1}^*$
\end{lemma}
\begin{proof}
We note that 
\[ \zeta_{i-1}  \zeta_i = Q_{i-2}^{-1}  \boundary_{i-1}  Q_{i-1}  Q_{i-1}^{-1}  \boundary_i  Q_i = Q_{i-2}^{-1}  \boundary_{i-1}  \boundary_i  Q_i.\]
As $\boundary_{i-1} \boundary_i = 0$, it immediately follows that $\zeta_{i-1} \zeta_i = 0$.  A similar argument holds for $\zeta_i^* \zeta_{i+1}^*.$
\end{proof}

At this point, we have a natural generalization of the Hodge Laplacian which incorporates inner product spaces, namely
\[ \zeta_i^*\zeta_i + \zeta_{i+1}\zeta_{i+1}^* = Q_i^{-1} \boundary_i^* \boundary_i  Q_{i} + Q_i^{-1}  \boundary_{i+1}  \boundary_{i+1}^* Q_{i}.\] 
However, on further inspection this generalization has some undesirable properties; it is, in general, not symmetric and hence non-Hermitian, because of the lack of symmetry, even in the graph case there is no choice of inner product which recovers the normalized Laplacian, and finally, as $\zeta^*_i \zeta_i + \zeta_{i+1}\zeta_{i+1}^* = Q_i^{-1} \mathcal{L}_i Q_i$, it is similar to and has the same spectrum as the standard Hodge Laplacian and hence the inner product structure has minimal impact.  

We note however, that the matrix adjoint and the matrix transpose coincide for the standard inner product space, while they do not coincide for an arbitrary inner product space.  Thus, interpreting the Hodge Laplacian instead as 
\[ \mathcal{L}_i = \boundary_i^T\boundary_i + \boundary_{i+1}\boundary_{i+1}^T \] yields \[Q_i  \boundary_i^*  Q_{i-1}^{-2}  \boundary_i  Q_i +Q_{i}^{-1}  \boundary_{i+1}  Q_{i+1}^{2}  \boundary_{i+1}^*  Q_i^{-1} \] as a potential inner product Laplacian.  We note that this matrix is Hermitian and positive semi-definite and is not similar to the standard Hodge Laplacian. Additionally, by \Cref{L:doubleB} this framing aligns with the linear algebraic viewpoint on the Hodge Lapalican as discussed by Lim~\cite{lim2020hodge}.  Thus, we focus the remainder of this work on the following choice of Laplacian for a simplicial complex equipped with inner product spaces.

\begin{define}[Inner Product (Hodge) Laplacian] 
Let $K$ be an abstract simplicial complex of dimension $k$.  Let $\set{\boundary_i}$ be the boundary maps between the chain complexes and let   $\set{Q_i}_{i=0}^{k}$ be a collection of symmetric positive definite matrices over the elements of the chain complexes $C_i(K,\R)$.  For the collection of inner product spaces defined by $\set{Q_i^2}$ the  $i^{\textrm{th}}$ \emph{inner product (Hodge) Laplacian} is defined as 
\[ \IPL[i] = Q_i  \boundary_i^*  Q_{i-1}^{-2} \boundary_i  Q_i +Q_{i}^{-1} \boundary_{i+1}  Q_{i+1}^{2} \boundary_{i+1}^*  Q_i^{-1}.\]
\end{define}

It is worth noting that if we consider the Hodge Laplacian as defined in terms of the co-boundary operators, i.e. $\mathcal{L}_i = \boundary_i^* \paren{\boundary_i^*}^T + \paren{\boundary_{i+1}^*}^T\boundary_{i+1}^*$, the resulting inner product Laplacian would be defined by 
\[ Q_i^{-1}  \boundary_i^* Q_{i-1}^2  \boundary_i  Q_{i}^{-1} + Q_i  \boundary_{i+1} Q_{i+1}^{-2}  \boundary_{i+1}^*  Q_i.\]
As this is simply $\IPL[i]$ for the inner products given by $\set{Q_i^{-2}}$ (as opposed to $\set{Q_i^2}$), there is no essential loss of generality in focusing only on the inner product Laplacian generated by boundary maps.  

Perhaps unsurprisingly, the inner product Laplacian also yields a Hodge decomposition for $\R^n$.  More concretely, we have:
\begin{lemma}
Let $K$ be a $k$-dimensional simplicial complex, let $\set{\boundary_i}_{i=1}^{k}$ define the boundary maps of the chain complex, and let $\set{Q_i^2}_{i=0}^{k}$ define inner product spaces on each of the chain groups.  If the set of $i$-faces of $K$ has size $n$, then 
\[ \R^n = \IM\paren{Q_i^{-1}\boundary_i^*Q_{i-1}} \oplus \Ker\paren{\IPL[i]} \oplus \IM\paren{Q_i\boundary_{i+1}Q_{i+1}^{-1}}.\]
\end{lemma}

\begin{proof}
Noting that, by \Cref{L:doubleB}, we have that
\[ \paren{Q_i^{-1} \boundary_i^* Q_{i-1}}^T \paren{Q_i \boundary_{i+1}Q_{i+1}^{-1}} = \zeta_i \zeta_{i+1} = 0.\] 
Thus $\IM\paren{Q_i^{-1}\boundary_i^*Q_{i-1}}$ and $\IM\paren{Q_i\boundary_{i+1}Q_{i+1}^{-1}}$ are orthogonal subspaces, and thus it remains to show the set of vectors orthogonal to $\IM\paren{Q_i^{-1}\boundary_i^*Q_{i-1}} \oplus \IM\paren{Q_i\boundary_{i+1}Q_{i+1}^{-1}}$ is exactly the kernel of $\IPL[i]$. To that end, we observe that 
\begin{align*} \Ker\paren{\IPL[i]} &= \Ker\paren{\zeta_i^T\zeta_i + \zeta_{i+1}\zeta_{i+1}^T}
= \Ker(\zeta_i) \cap \Ker(\zeta_{i+1}^T) \\ &= \IM\paren{\zeta_i^T}^{\perp} \cap \IM\paren{\zeta_i}^{\perp} = \IM\paren{Q_i^{-1}\boundary_i^*Q_{i-1}}^{\perp} \cap \IM\paren{Q_i\boundary_{i+1}Q_{i+1}^{-1}}^{\perp}, 
\end{align*}
as desired.
\end{proof}

As an immediate consequence, the non-zero spectrum for $\IPL[i]$ can be calculated independently from the non-zero spectra of $Q_i  \boundary_i^*  Q_{i-1}^{-2} \boundary_i  Q_i$  and $Q_{i}^{-1} \boundary_{i+1}  Q_{i+1}^{2} \boundary_{i+1}^*  Q_i^{-1}$.  This immediately motivates the following definition.

\begin{define}[(Boundary) Semi-Hodge Laplacian]
Let $V$ and $E$ be finite sets, with associated inner products defined by $M_V = Q_V^2$ and $M_E = Q_E^2$ and let $B \in \set{-1,0,1}^{\size{V} \times \size{E}}$.  We define the \emph{semi-Hodge Laplacian} for $B$, $M_V$ and $M_E$ by $\sHL = Q_{V}^{-1}B M_E B^T Q_{V}^{-1}$.  If $B$ represents the signed boundary matrix from a simplicial complex, we will refer to this matrix as a \emph{boundary semi-Hodge Laplacian}. 
\end{define}

We note that the entry-wise absolute value of $B$ may be thought of as the incidence matrix of a hypergraph with vertex set $V$ and edge-set $E$, which we will denote by $\mathcal{H}_B$.  We will thus abuse notation and refer to properties of $B$ by the appropriate hypergraph terminology.  For instance, for a fixed $v \in V$ the set $e \in E$ such that $B_{ve} \neq 0$ are the edges incident to $v$ and are denoted $E(v)$.  Thus, the degree of $v$ is $\size{E(v)}$ and $\norm[1]{B}$ is the maximum degree of $\mathcal{H}_B$, which we will denote by $\Delta_B$.  Now if $V$ and $E$ are equipped with inner product spaces, we can think of $\mathcal{H}_B$ as a form of a weighted hypergraph where the weight of a set $S$ is given by the norm-squared in the appropriate inner product space.  More concretely, if we let $M_E$ and $M_V$ denote the respective inner product spaces of $E$ and $V$, then for $e \in E$, we can define $w_e = \one_e^TM_E\one_e$ and similarly for $v \in V$.  While we employ the standard notational abuse where $w_S = \sum_{s \in S} w_s$, we caution that, in general (i.e., if the weak non-conformality of the inner product space is not zero), $w_S \neq \one_S^TM\one_S$.  As it will be useful to reference the ``size" of the edges incident with a vertex with respect to the inner product space on the edges, we will denote $\one_{E(v)}^TM_E\one_{E(v)}$ by $d_v$, aligning with standard notation for the degree of a vertex while reserving $\deg(v)$ for the combinatorial degree of a vertex, i.e.\ $\deg(v) = \size{E(v)}$.

Since the semi-Hodge Laplacian depends on the both $M_V^{-1}$ and $M_E$ it is necessary to quantify the difference in scales between $M_V$ and $M_E$.  One natural approach is to consider the ratio between the spectral radii, which immediately yields an upper bound based on the maximum degree of the underlying hypergraph.  However, by incorporating the hypergraph structure into our definition, we can develop a more refined upper bound on the spectral radius of the semi-Hodge Laplacian.  To that end, we define the compataiblity between two inner product spaces which defines the scale based on the neighborhood structure of the hypergraph.
\begin{define}[Compatibility]
If $M_V$ and $M_E$ are inner products spaces on the vertices and edges of a (hyper)graph $G = (V,E)$, we will say that the inner product spaces are \emph{$\omega$-compatible} if for all $v \in V$  $\one_{E(v)}^TM_E\one_{E(v)} \leq \omega \one_v^TM_V\one_V$.  Furthermore, if this inequality is tight for all $v \in V$, we will say that $M_E$ and $M_V$ are \emph{perfectly $\omega$-compatible}.
\end{define}

With this notation in hand, we can provide a upper bound on the spectral radius of the semi-Hodge Laplacian.  It is worth noting that when applied to the combinatorial and normalized Laplacians (and their weighted variants) this result gives the standard easy upper bounds on the spectral radius~\cite{brouwer2011spectra,chung1997spectral}.  However, while the standard proofs of these upper bounds rely on the Gershgorin disk theorem to contain the spectrum, our approach exploits the compatibility between the edge and vertex inner product spaces.

\begin{lemma}\label{L:radius}
Let $V$ and $E$ be finite sets with associated inner products $M_V = Q_V^2$ and $M_E = Q_E^2$ and let $B \in \set{-1,0,1}^{\size{V} \times \size{E}}$.  If $M_V$ and $M_E$ are $\omega$-compatible inner product spaces which are weakly $\rho_V$- and $\rho_E$-conformal, respectively, then the semi-Hodge Laplacian has spectral radius at most 
\[  
\frac{1+\rho_V}{1-\rho_V}\paren{\frac{1+\rho_E}{1-\rho_E}}^2 r \omega\]
where $r$ is the rank of $\mathcal{H}_B$.
\end{lemma}

\begin{proof}
    We note that since $\sHL$ is positive semi-definite, the spectral norm corresponds to the maximum of the Rayleigh quotient.  Since both $M_V$ and $M_E$ are weakly-$\rho$ conformal, we may apply \Cref{L:trace} to bound the Rayleigh quotient.  More concretely, we have 
    \[ \norm{\sHL} = \max_{z \neq 0} \frac{z^T M_V^{-\half} B^T M_E B M_V^{-\half} z}{z^Tz} 
    = \max_{\widehat{z} \neq 0} \frac{ \widehat{z}^T B^T M_E B \widehat{z}}{\widehat{z}^T M_V \widehat{z}}. \]
    By applying \Cref{L:trace}, the Cauchy-Schwarz inequality, and \Cref{L:trace} again, we have
    \begin{align*}
        \widehat{z}^T B^T M_E B \widehat{z} &\leq \frac{1+\rho_E}{1-\rho_E} \sum_{e} w_e \sum_{v \in e} (B_{ev}\widehat{z}_v)^2 \\ 
        &\leq \frac{1+\rho_E}{1-\rho_E} \sum_{e} w_e r\sum_{v \in e} \widehat{z}_v^2 \\
        &= r \frac{1+\rho_E}{1-\rho_E} \sum_v \widehat{z}_v^2\sum_{e \in E(v)} w_e \\ 
        &\leq r \paren{\frac{1+\rho_E}{1-\rho_E}}^2 \sum_v \widehat{z}_v^2 \one_{E(v)}^TM_E\one_{E(v)}.
    \end{align*} 
    Now applying \Cref{L:trace} to $\widehat{z}^TM_V\widehat{z}$, we can bound the denominator of the Rayleigh quotient below by $\frac{1-\rho_V}{1+\rho_V} \sum_{v} \widehat{z}_v^2 \one_v^TM_V\one_v$.  Observing that by the $\omega$-compatability of $M_V$ and $M_E$, $\one_{E(v)}^TM_E\one_{E(v)} \leq \omega \one_v^TM_V\one_v$, completes the proof.
\end{proof}

\hide{
\begin{proof}
  We note that since $\sHL$ is positive semi-definite, the spectral norm corresponds to the maximum of the Rayleigh quotient.  Furthermore, since both $M_V$ and $M_W$ are weakly-$\rho$ conformal, we may apply \Cref{L:trace} to simplify the Rayleigh quotient.  More concretely, we have 
\[ 
    \norm{\sHL} = \max_{z \neq 0} \frac{z^T M_V^{-\half} B^T M_E B M_V^{-\half} z}{z^Tz} 
    = \max_{\widehat{z} \neq 0} \frac{ \widehat{z}^T B^T M_E B \widehat{z}}{\widehat{z}^T M_V \widehat{z}}. \]
    Let $\set{s_e}$ be the set of basis vectors indexed by $e \in E$, then $B\widehat{z} = \sum_{e \in E} \sum_{v \in E} (B_{ev} \widehat{z}_v) s_e$ and so, by the triangle inequality and Cauchy-Schwarz
    \[ \widehat{z}^T B^T M_E B \widehat{z} \leq \sum_{e} \sqrt{ \paren{ \sum_{v \in e} B_{ev}\widehat{z}_v}^2 w_e} \leq \sum_e \sqrt{ r \paren{ \sum_{v \in e} \widehat{z}_v^2} w_e}. \]
    
    For the moment, suppose that $\sum_{v \in e} \widehat{z}_v^2 > 0$ for all $e$. Since the Rayleigh quotient is invariant under scaling, without loss of generality we may scale up $\widehat{z}$ such that $r \paren{\sum_{v \in e} \widehat{z}_v^2} s_e^TM_Es_e > 1$ for all $z$.  Thus, we have that
     \[ \widehat{z}^T B^T M_E B \widehat{z} \leq \sum_e r \paren{ \sum_{v \in e} \widehat{z}_v^2} s_e^TM_Es_e = r \sum_v \widehat{z}_v^2 \sum_{e \in E(v)} w_e \leq r \sum_v \widehat{z}_v^2 \frac{1+\rho_E}{1-\rho_E} \one_{E(v)}^TM_E\one_{E(v)}, \]
    where the last inequality follows by \Cref{L:trace}.  Now noting that \Cref{L:trace} also gives that  $z^TM_vz \geq \frac{1-\rho_V}{1+\rho_V} \widehat{z}_v^2 w_v$, combined with the $\omega$-compatibility of $M_V$ and $M_E$ immediately gives the desired result.  

    Now we turn to the case where $\sum_{v \in e} \widehat{z}_v^2 = 0$ for some edge $e$.  Note that this implies that $\widehat{z}_v = 0$ for all $v \in e$ and in particular, the corresponding term $\widehat{z}_v^2 w_v$ in the denominator of the Rayleigh quotient is zero.  Thus, we may restrict our attention to the support of $\widehat{z}$ and the above argument proceeds in a similar manner.
\end{proof}
}

\section{Existing Laplacians and the Inner Product Laplacian}
As briefly noted earlier, the normalized and combinatorial Laplacian and their weighted variants can be easily recovered by the appropriate choice of vertex and edge inner product.  Concretely, if the edge inner product matrix is the diagonal matrix of edge weights (which are one if the graph is unweighted) then using the identity or the diagonal weighted degree matrix for the vertex inner product, yields the combinatorial and normalized Laplacians, respectively.  The compatibility for these inner products is given by the maximum weighted degree and 1, and so \Cref{L:radius} yields the standard bounds on the spectral radius of the Laplacian.  It worth noting that while the signless Laplacian ($D + A$) and its normalized variant ($I + D^{-\half}AD^{-\half}$) are not obviously inner product Laplacians, they are given by a semi-Hodge Laplacian.  Specifically, if $H$ is the (unsigned) incidence matrix for a graph and $D$ is the diagonal matrix of degrees, the signless Laplacian is $HH^T$ and the normalized signless Laplacian is $D^{-\half}HH^TD^{-\half}$.  Again, the results provided by \Cref{L:radius} agree with the standard results bounding the spectral radius.  

\subsection{Laplacians for Higher-Order Structures}\label{SS:higher_order}

\hide{\sinan{List of hypergraph Laplacians for brainstorming. Let $H$ denote the incidence matrix:
\begin{enumerate}
    \item Signless Hypergraph Laplacian (Cardoso): $HH^T$
    \item Rodri guez's Laplacian: $D_v-HH^T$.
    \item Bolla's Laplacian: $D_v -HD_e^{-1}H^T$
    \item Zhou: $I-\frac{1}{2}D_v^{-1/2}HWH^TD_v^{-1/2}$ where $W$ is the diagonal matrix of given hyperedge weights.
    \item Mulas's hypergraph Laplacian for oriented hypergraphs: $I-D^{-1/2}AD^{-1/2}$ where $A_{ij}$ is number of hyperedges in which $i,j$ are anti-oriented minus number which are co-oriented.
    \item Carletti's RW Laplacian: see paper, horrid notation. 
    \item Chung's hypergraph Laplacian
    \item Hypergraph Laplacian based on Fan's digraph Laplacians
    \item Hypergraph Magnetic Laplacian
\end{enumerate}

We concluded worth mentioning Fan's, Mulas' and Bollas'. 
Point: if/how these can framed as a special case of IPL. 
}}

Over the last several decades there have been a variety of proposed methods to generalize the framework of spectral graph theory to the context of higher-order structures such as hypergraphs or simplicial complexes.  Broadly speaking, these fall into one of two classes: tensor-based methods and matrix-based methods.  While the tensor-based approach has significant theoretical appeal, the tensor spectral theory is significantly more complicated even in the case of tensors with significant symmetries~\cite{cooper2020adjacency, cooper2012spectra}. Thus, as a practical matter, there has been significantly more work around generating matrix representations of hypergraphs and simplicial complexes, including \cite{Bolla:Hypergraph, Carletti:RWhypergraph, chung1992laplacian, Horak:WeightedLaplcian, Mulas:Oriented,mulas2021spectral,Rodriguez:hypergraph,Zhou:hypergraph} and many others.  However,  many of these higher-order matrix-generalizations of both the normalized and combinatorial Laplacian can be interpreted as an inner product Laplacian with appropriate choices of inner products. 
For example, and as noted earlier, the weighted simplicial complex Laplacian of Horak and Jost~\cite{Horak:WeightedLaplcian} can be thought of as an inner product Laplacian when appropriately symmetrized.  Similarly, the Laplacian defined by Chung for $k$-uniform hypergraphs~\cite{chung1992laplacian}, while initially defined based on the combinatorial properties of the $k-1$ sets contained within the edges, can also be defined as a positive weighted combination of the up- and down-Laplacians for the $k-2$ faces in the simplicial complex formed by interpretting the maximal hyperedges as facets.  

The approach of Horak and Jost, as well as Chung, to hypergraph Laplacians is somewhat of an exception in that the Laplacian matrix is indexed by something other than the vertices of the hypergraph.  For example, Mulas and Jost~\cite{Mulas:Oriented}, introduce a hypergraph Laplacian for oriented hypergraphs $D^{-1}HH^T$ where $D$ is the diagonal matrix of vertex degrees and $H$ is the vertex-edge incidence matrix for the hypergraph where the signing of  $H$ corresponds to the input/output of the oriented hyperedge.  The appropriate symmetrization of this matrix is easily seen to be a semi-Hodge Laplacian.\footnote{We note that, as an operator, the Laplacian of Jost and Mulas is self-adjoint and that the symmetrization operation amounts to rescaling the basis which is used to provide the matrix representation.} More commonly, the vertex-edge incidence matrix is chosen to be unsigned and a Laplacian is defined by mimicking the form of the combinatorial Laplacian, that is $D - \tilde{D}HWH^T\tilde{D}$ where $D$, $\tilde{D}$, and $W$ are appropriately sized diagonal matrices with strictly positive diagonal.  The Laplacians of Bolla~\cite{Bolla:Hypergraph}, Rodriguez~\cite{Rodriguez:hypergraph}, and Zhou, Huang, and Sch{\"o}lkopf~\cite{Zhou:hypergraph} as well as symmetric random walk Laplacians of Carletti, Battiston, Cencetti, and Fanelli~\cite{Carletti:RWhypergraph} all fall into this general framework.  While not directly expressible as an inner product Laplacian or a semi-Hodge Laplacian because of the subtraction, we can see that these all correspond to an inner product Laplacian on the clique expansion of the hypergraph.

\begin{lemma}\label{L:hyper}
Let $\mathcal{H} = (V,E)$ be a hypergraph on $n$ vertices and $m$ edges.  Let $H \in \set{0,1}^{n \times m}$ be the vertex-edge incidence matrix for $\mathcal{H}$.  Suppose $L$ has the form $D - \tilde{D}HWH^T\tilde{D}$, where $D$, $\tilde{D}$, and $W$ are appropriately sized diagonal matrices with strictly positive diagonals. If there exists a strictly positive vector $\pi$ in the kernel of $L$, then $L$ can be represented as an inner product Laplacian for the clique expansion of $\mathcal{H}$.
\end{lemma}

\begin{proof}
    For convenience of notation, we will let $d$, $\tilde{d}$, and $w$ be the vectors associated with the diagonals of $D$, $\tilde{D}$, and $W$, respectively.  Similarly, we will define $\Pi$ to be the diagonal matrix associated with $\pi$. Now consider the matrix $\tilde{L} = \Pi L \Pi$.  As $\Pi$ is diagonal, in order to show that $L$ is an inner product Laplacian for the clique expansion of $\mathcal{H}$, it suffices to instead show that $\tilde{L}$ is an inner product Laplacian for the clique expansion of $\mathcal{H}$. To that end, we first note that for $u \neq v$
    \[ \tilde{L}_{u,v} = - \pi_u \pi_v \tilde{d}_u\tilde{d}_v \sum_{e \in E} H_{u,e} w_e H_{v,e} = \tilde{w}_{\set{u,v}}. \]
    Thus $\tilde{L}_{u,v} \neq 0$ if and only if $\set{u,v}$ is in the clique expansion of $\mathcal{H}$.  Thus if $\tilde{H}$ is the signed boundary matrix of $\mathcal{H}$ and $\tilde{W}$ is an appropriately ordered diagonal matrix of the $w_{\set{u,v}}$, then $\tilde{L} - \tilde{H} \tilde{W} \tilde{H}^T$ is a diagonal matrix.  Furthermore, since by construction $\tilde{H}\tilde{W}\tilde{H}^T\one = 0$, and by definition $\tilde{L}\one = 0$, we have that $\tilde{L} - \tilde{H}\tilde{W}\tilde{H}^T = 0$.  This implies that $\tilde{L}$, and hence $L$, is an inner product Laplacian for the clique expansion of $\mathcal{H}$ as desired.
\end{proof}

The observation that many proposed vertex-indexed hypergraph Laplacians are either combinatorial or normalized Laplacians of the hypergraph's weighted clique expansion is not new~\cite{agarwal2006higher, agarwal2005beyond}. Indeed, Chitra and Raphael point to this observation to inspire their definition of ``edge-dependent vertex weighted" random walks as a means of capturing more hypergraph structure through non-reversible Markov chains~\cite{chitra2019random}.  Unsurprisingly, their approach yields a positive transition probability between $u$ and $v$ whenever there is an edge containing both $u$ and $v$, that is, whenever $\set{u,v}$ is in the clique expansion of $H$.  However, as noted in \cite{hayashi2020hypergraph}, the transition probability matrix for the edge-dependent vertex weighted random walks can be transformed into a Laplacian-like matrix via the combinatorial and normalized digraph Laplacians defined by Chung~\cite{chung2005laplacians}.  A similar approach as \Cref{L:hyper} yields that these digraph Laplacians can also be thought of as inner product Laplacians of the undirected graph formed by the relation $\set{u,v}$ is an edge if and only if $P_{uv} > 0$ or $P_{vu} > 0$.  In the case of the edge-dependent vertex weighted random walks, this undirected graph is precisely the clique expansion of the original hypergraph.

\begin{lemma}\label{L:digraph_lap}
    Let $P$ be the transition probability matrix for an ergodic Markov chain and $\Pi$ be the diagonal matrix associated with the limiting distribution.  The digraph Laplacian $L = \Pi - \frac{1}{2}\paren{\Pi P + P^T \Pi}$ and the normalized digraph Laplacian $\mathcal{L} = I - \frac{1}{2}\paren{\Pi^{\half}P\Pi^{-\half} + \Pi^{-\half}P^T\Pi^{\half}}$ are both inner product Laplacians on the associated undirected graph.
\end{lemma}

Finally, it is worth noting that while the signless hypergraph Laplacian, defined by Cardoso and Trevisan as $HH^T$ where $H$ is the (unsigned) vertex-edge incidence matrix~\cite{Cardoso:hypergraph}, is not an inner product Laplacian, much like the signless Laplacian it is trivially a semi-Hodge Laplacian for the hypergraph.

\subsection{Laplacian-like Spectra for Subgraphs}\label{SS:subgraphs}
While spectral graph theory approaches typically consider the spectrum of the entire graph, there are a few approaches which attempt to characterize the spectral behavior of a (induced) subgraph of the graph, viewing the entire graph as an ambient space or manifold which contains the (induced) subgraph of interest.  As in the continuous case, the conditions imposed on the boundary have significant impacts on the behavior of the eigenvalues and eigenfunctions on the subgraph.  While arbitrary boundary conditions could be imposed, the most well known boundary conditions are the Dirichlet and Neumann boundary conditions which specify the behavior of the function and its derivative, respectively, at the boundary. In the case of a graph, the derivative at the boundary is defined as the sum of the finite differences for edges incident to the vertex and crossing the boundary.  Letting $\partial S$ denote the vertex boundary of $S$, that is, the set of vertices which neighbor some vertex in $S$ but are not contained in $S$, and letting $E_S$ denote the edge contained in $S$ together with the edge boundary $E(S,\overline{S})$, we have the following definitions for Dirichlet and Neumann eignevalues of an induced subgraph.
\begin{define}[Dirichlet Eigenvalues]
Let $G=(V,E)$ be a graph and let $S \subseteq V$ be such that there is a nonempty vertex boundary $\partial S$. A function $f$ is said to satisfy the Dirichlet boundary condition if $f(x) = 0$ for all $x \in \partial S$ in which case we write $f \in D^*$. The \emph{Dirichlet eigenvalues} of the induced subgraph on $S$ are given by the relations:
\begin{align*}
\lambda_1^{(D)} &= \inf_{\substack{f \neq 0 \\  f\in D^*}} \frac{\sum_{\{u,v\} \in E_S}(f(u)-f(v))^2}{\sum_{u \in S} f(u)^2\deg{(u)}},
\end{align*}
and in general
\begin{align*}
    \lambda_i^{(D)} &= \inf_{\substack{f \neq 0 \\ f \in D^*}} \sup_{f' \in C_{i-1}}\frac{\sum_{\{u,v\} \in E_S}(f(u)-f(v))^2}{\sum_{u \in S} (f(u)-f'(u))^2\deg{(u)}}, 
\end{align*}
where $C_i$ is the subspace spanning by eigenfunctions $\phi_j$ achieving $\lambda_j$ for $1\leq j \leq i \leq \size{S}$.
\end{define}

\begin{define}[Neumann Eigenvalues]
Let $G = (V,E)$ be a connected graph and let $S$ be a subset of at least two vertices with non-empty vertex boundary $\partial S$.  The \emph{Neumann eigenvalue} for $G$ and $S$ is given by 
\[ \lambda_S = \min_{f} \frac{ \sum_{\set{u,v} \in E_S} (f(u)- f(v))^2}{\sum_{s \in S} f(s)^2 \deg(s)} \]
where $f \colon S \cup \partial S \rightarrow \R$ such that $\sum_{s \in S} f(s) \deg(s) = 0$.
\end{define}

We note that that the Neumann boundary condition is implicitly enforced in this definition.  More concretely, as the vertices in $\partial S$ do not appear the denominator, to achieve the infimum it suffices to minimize $\sum_{u \in E(v,S)} (f(u) - f(v))^2$ for each $v \in \partial S$, independently. Since this is a sum of squares, the minimum occurs precisely when $\sum_{u \in E(v,S)} 2 (f(v) - f(u)) = 0$, that is, when the derivative at $v$ is 0.  It is also worth noting that the Neumann eigenvalue has an ``orthogonality" constraint imposed on the domain of optimization, namely $\sum_{s \in S} \deg(s)f(s) = 0.$ 

In contrast, the Dirichlet condition is explicitly enforced and has no orthogonality constraint imposed on the domain of integration.  Because the constraint that the boundary is zero is explicitly enforced,  the numerator can be decomposed into two terms, 
\[
\sum_{\set{u,v} \in E(S,S)} \paren{f(u)-f(v)}^2 \mbox{  and a term of the form  } \sum_{v \in S} e(\set{v},\overline{S})f(v)^2.
\]
Note that the second term is zero only if $f(v) = 0$ for all $v$ which border $\partial S$, i.e. $\partial (\partial S)$.  Applying this decomposition recursively (essentially peeling the $S$ from the outside) yields that the numerator is 0 only if $f(s) = 0$ for all $s \in S$, and in particular, $\lambda^{(D)}_i > 0$ for all $i$. Combining this observation with the lack of an orthogonality constraint indicates that the Dirichlet spectrum can not be directly recovered from a inner product Laplacian on the graph $G$.  In particular, since every inner product Laplacian has a non-trivial eigenspace associated with the eigenvalue 0, this implies that to represent $\lambda_1^{(D)} > 1$ as an element of the spectrum of an inner product Laplacian, it must be necessary that the domain of optimization be restricted to functions orthogonal to a particular subspace.  

On the other hand, the Neumann eigenvalue (and the associated eigenfunction) can be recovered as the limit of the spectrum of a sequence of inner product Laplacians.

\begin{theorem}\label{T:neumann}
Let $G = (V,E)$ be a connected graph and $\lambda_S$ denote its Neumann eigenvalue, for a subset $S$ with at least two vertices.  There exists a sequence of inner products defined on the vertices and edges of $G$, $\set{(M_V^{(i)},M_E^{(i)})}$, such that for the associated inner product Laplacians $\set{\IPL[(i)]}$:
\begin{itemize}
    \item all have a one-dimensional zero-eigenspace,
    \item the smallest non-zero eigenvalues, $\set{\lambda_2^{(i)}}$, converge to $\lambda_S$, and
    \item there exist harmonic eigenvectors $\set{f^{(i)}}$ associated with $\lambda_1^{(i)}$, where $f^{(i)}$ converges point-wise to some function $f$ which realizes the Neumann eigenfunction.
\end{itemize}
\end{theorem}

\begin{proof}
In order to construct the sequence $\set{(M_V^{(i)},M_E^{(i)})}$, we first define a pair $(M_V^{(\epsilon)},M_E^{(\epsilon)})$ for all $\epsilon > 0$ and the desired sequence will be defined by a subsequence $\epsilon_i$ which monotonically converges to $0$.  In order to define $M_E^{(\epsilon)}$, we first define the graph $G_{\epsilon} = (V,E,w_{\epsilon})$ where the weight function is $1$ for all edges in $E_S$ and $\epsilon$ for $E - E_S$.  We then define $M_E^{(\epsilon)}$ as the diagonal matrix given by the weight function $w_{\epsilon}.$  To emphasize the weighting of the edges, we will denote by $E_1 = E_S$ the set of edges of weight 1 and denote by $E_{\epsilon} = E - E_S$ the set of edges with weight $\epsilon$.  For any vertex $v \in V$ we will denote by $\deg_{\epsilon}(v)$ the weighted degree of $v$ in $G_\epsilon$ and define an inner product on the vertices by $M_V^{(\epsilon)} = \paren{Q_V^{(\epsilon)}}^2$ where 
\[ \left(Q_V^{(\epsilon)}\right)_{u,v} = \begin{cases} \sqrt{\deg_{\epsilon}(v)} & u = v \in S \\ \sqrt{\epsilon \deg_{\epsilon}(v)} & u = v \not\in S \\
0 & u \neq v \end{cases}.\]
We emphasize that $M_V^{(\epsilon)}$ is not the diagonal weighted degree matrix for $G_{\epsilon}$, but rather the diagonal weighted degree matrix for $G_{\epsilon}$ with the weight of vertices in $\overline{S}$ scaled by $\epsilon$. 

  Since $G$ is connected the zero eigenspace for each $\IPL[\epsilon]$ has dimension 1 and is given by $Q_V^{(\epsilon)}\one$.  Denote by $\lambda_{\epsilon}$ the second smallest eigenvalue of $\IPL[\epsilon]$ and let $g_{\epsilon}$ be an associated eigenvector with unit norm.   We note that, using the Rayleigh quotient, we may bound $\lambda_{\epsilon}$ above by a constant independent of $\epsilon$.  In particular, fix distinct $s, t \in S$ and consider the vector $h$ such that $h_s = \sqrt{\deg(t)}$, $h_t = -\sqrt{\deg(s)}$, and all other entries are 0.  Since for any $r \in S$, $\deg_{\epsilon}(r) = \deg(r)$, we have that $h^T Q_V^{(\epsilon)} \one = h_s \sqrt{\deg_{\epsilon}(s)} + h_t \sqrt{\deg_{\epsilon}(t)} = 0$ and $h^T\IPL[\epsilon]h = \deg(s) + \deg(t) - 2\one_{s \sim t}.$  Thus \[\lambda_{\epsilon} \leq \frac{\deg(s) + \deg(t) - 2\one_{s \sim t}}{\deg(s) + \deg(t)} \leq 1.\] As a consequence, the pairs $\set{(g_{\epsilon},\lambda_{\epsilon})}_{\epsilon > 0}$ are contained in a compact set and so there is a limit point.  More concretely, there exists a pair $(g,\lambda)$ and a decreasing sequence $\epsilon_1,\epsilon_2,\ldots$  such that $\lim_{i \rightarrow \infty} \epsilon_i = 0$ and $\lim_{i \rightarrow \infty}  (g_{\epsilon_i}, \lambda_{\epsilon_i}) = (g,\lambda).$  From this point forward, we will abuse notation and use $\epsilon \rightarrow 0^+$ to denote convergence along an appropriately chosen subsequence of $\set{\epsilon_i}$.

Now, in order to recover the Neumann eigenvalue and eigenvector from $(g_{\epsilon},\lambda_{\epsilon})$, we define $f_{\epsilon}$ by the relation $Q_V^{(\epsilon)}f_{\epsilon} = g_{\epsilon}$. By standard transformations of the Rayleigh quotient, we note that each $f_{\epsilon}$ satisfies
\[ \lambda_{\epsilon} = \sum_{\set{u,v} \in E_1} \paren{f_{\epsilon}(u) - f_{\epsilon}(v)}^2 + \epsilon \sum_{\set{u,v} \in E_{\epsilon}} \paren{f_{\epsilon}(u) - f_{\epsilon}(v)}^2.\] 

Combining this relationship with the unit norm of $g_{\epsilon}$ provides natural bounds on the entries of $f_{\epsilon}$.  In particular, for $s \in S$ we have that 
\[ \abs{f_{\epsilon}(s)} = \abs{ \frac{1}{\sqrt{\deg_{\epsilon}(s)}}g_{\epsilon}(s)} = \frac{1}{\sqrt{\deg(s)}}\abs{g_{\epsilon}(s)} \leq \frac{1}{\sqrt{\deg(s)}}.\] Further, for $t \in \partial S$ with neighbor $s \in S$ we have that $\paren{f_{\epsilon}(s) - f_{\epsilon}(t)}^2 \leq \lambda_{\epsilon} \leq 1$ and hence $\abs{f_{\epsilon}(t)} \leq 1 + \abs{f_{\epsilon}(s)}$.  In a similar fashion, consider $t \in \partial(S \cup \partial S)$ and let $s$ be a neighbor of $t$ in $S \cup \partial S$.  We then have that $\abs{f_{\epsilon}(t)} \leq \nicefrac{1}{\sqrt{\epsilon}} + \abs{f_{\epsilon}(s)} \in \bigOh{\nicefrac{1}{\sqrt{\epsilon}}}.$  Proceeding inductively, we can conclude that $f_{\epsilon}(t) \in \bigOh{\nicefrac{1}{\sqrt{\epsilon}}}$ for all $t \not\in S \cup \partial S.$  As a consequence, we can restrict the sequence $\set{\epsilon_i}$ to a subsequence such that $f_{\epsilon_i^{\prime}}(s)$ converges for all $s$ in $S \cup \partial S$ and such there is a universal constant $C$ such that $f_{\epsilon_i^{\prime}}(t) \leq \nicefrac{C}{\sqrt{\epsilon_i^{\prime}}}$ for all $t \not\in S \cup \partial S$. We will denote the limit on $S \cup \partial S$ as $f \colon S \cup \partial S \rightarrow \R$.  For convenience, we define $\lambda_f$ as the Rayleigh quotient associated with $f$, that is 
\[ \lambda_f = \frac{\sum_{\set{u,v} \in E_S} (f(u) - f(v))^2}{\sum_{u \in S} f^2(u)\deg(u)}.\]
At this point, to complete the proof it suffices to show that $f$ is in the domain of optimization for the Neunmann eigenvalue and that that $\lambda_f = \lambda = \lambda_S$.

For the the first, we note that 
\[ \abs{\sum_{s \in S} f_{\epsilon}(s)\deg(s)} = \abs{\sum_{s \in S} f_{\epsilon}(s) \deg_{\epsilon}(s) } = \abs{\sum_{s \in S} g_{\epsilon}(s) \sqrt{\deg_{\epsilon}(s)} } = \abs{\sum_{s \not\in S} g_{\epsilon}(s) \sqrt{\epsilon \deg_{\epsilon}(s)}} \leq  \sqrt{\epsilon \Vol(G_{\epsilon})} \longrightarrow 0,\]
where the last equality comes from the fact that $g_{\epsilon}^TQ_V^{(\epsilon)}\one = 0$ and the inequality follows from the Cauchy-Schwarz inequality. As $f_{\epsilon}(s) \rightarrow f(s)$ on $S$, we have that $\sum_{s \in S} f(s)\deg(s) = 0$.

To establish that $\lambda_S = \lambda_f = \lambda$ we will first show that $\lambda_S \leq \lambda_f \leq \lambda$ and then show that $\lambda_S \not< \lambda.$  As $f$ is in the domain of optimization for the Neumann eigenvalue, we have immediately that $\lambda_S \leq \lambda_f$.  To show that $\lambda_f \leq \lambda$, we first note that as $f_{\epsilon}(t) \in \bigOh{1}$ for any $t \in \partial S$, $f_{\epsilon}(t)^2 \epsilon \deg_{\epsilon}(t) \rightarrow 0$ as $\epsilon \rightarrow 0^+$.  Additionally, as $f_{\epsilon}(t) \in \bigOh{\nicefrac{1}{\sqrt{\epsilon}}}$ for $t \not\in \paren{\partial S \cup S}$, we have $f_{\epsilon}(t)^2 \epsilon \deg_{\epsilon}(t) \rightarrow 0$ as $\epsilon \rightarrow 0^+$ because $\deg_{\epsilon}(t)$ goes to zero like $\epsilon$.  Thus
\begin{align*}
    \lambda &= \lim_{\epsilon \rightarrow 0^+} \frac{ \sum_{\set{u,v} \in E_1} \paren{f_{\epsilon}(u) - f_{\epsilon}(v)}^2 + \epsilon \sum_{\set{u,v} \in E_{\epsilon}} \paren{f_{\epsilon}(u) - f_{\epsilon}(v)}^2}{\sum_{s \in S} f_{\epsilon}(s)^2\deg_{\epsilon}(s) + \sum_{t \not\in S} f_{\epsilon}(t)^2\epsilon \deg_{\epsilon}(t)}\\ 
    &\geq \lim_{\epsilon \rightarrow 0^+} \frac{ \sum_{\set{u,v} \in E_1} \paren{f_{\epsilon}(u) - f_{\epsilon}(v)}^2}{\sum_{s \in S} f_{\epsilon}(s)^2\deg_{\epsilon}(s) + \sum_{t \not\in S} f_{\epsilon}(t)^2\epsilon \deg_{\epsilon}(t)} \\
    &= \frac{\sum_{\set{u,v} \in E_1} (f(u) - f(v))^2}{\sum_{s \in S} f(s)^2\deg(u)} \\
    &= \lambda_f
\end{align*}   

To complete the proof, suppose that $\lambda_S < \lambda$ and let $f^* \colon S \cup \partial S \rightarrow \R$ be a vector achieving the Neumann eigenvalue such that $\sum_{s \in S} f^*(s)^2 \deg(s) = 1$.  By way of contradiction, we will define a family of functions $f_{\epsilon}^*$ such that $f_{\epsilon}^*$ is orthogonal to the principle eigenspace of $\IPL[\epsilon]$ for all $\epsilon$ and for all sufficiently small $\epsilon$ we have that
\[ \frac{\paren{f_{\epsilon}^*}^T \IPL[\epsilon] f_{\epsilon}^*}{\paren{f_{\epsilon}^*}^T M_V^{(\epsilon)} f_{\epsilon}^*} < \lambda_{\epsilon}.\] 
To this end,  we define  $C_{\epsilon} = \sum_{s \in \partial S} f^*(s) \epsilon \deg_{\epsilon}(S)$ and
\[ f_{\epsilon}^*(s) = \begin{cases}  f^*(s) - \frac{C_{\epsilon}}{\Vol(S)} & s \in S \\ f^*(s) & s \in \partial S \\ 0 & s \not\in S \cup \partial S\end{cases}. \]
It is easy to confirm that 
\begin{align*}
    \sum_{s \in S} f^*_{\epsilon}(s) \deg_{\epsilon}(s) + \sum_{t \not\in S} f_{\epsilon}^*(t) \epsilon \deg_{\epsilon}(t) & = 0. \\
    \lim_{\epsilon \rightarrow 0^+} \sum_{\set{u,v} \in E_1} \paren{f_{\epsilon}^*(u) - f_{\epsilon}^*(v)}^2 + \epsilon \sum_{\set{u,v} \in E_{\epsilon}} \paren{f_{\epsilon}^*(u) - f_{\epsilon}^*(v)}^2 &= \lambda_S, \quad \textrm{and} \\
    \lim_{\epsilon \rightarrow 0^+} \sum_{s \in S} f^*_{\epsilon}(s)^2\deg_{\epsilon}(s) + \sum_{t \in \overline{s}} (f^*_{\epsilon}(t))^2 \epsilon \deg_{\epsilon}(t) &= 1,
\end{align*}
yielding the desired contradiction.
\end{proof}

\section{Isoperimetric Inequalities for Inner Product Laplacians}\label{S:iso}
In this final section, we show how the standard isoperimentric results for combinatorial and normalized Laplacians \cite{brouwer2011spectra, chung1997spectral} can be generalized to the case of the inner product Laplacian.  As a consequence, we unify the theoretical discussion of isoperimentric inequalities for these two Laplacians and provide a general framework for understating isoperimetric inequalities in graph.  Of particular note, the generalization of the isoperimetric properties to the inner product Laplacian provides theoretical justification for the using the inner product Laplacian to perform spectral clustering.  Such spectral clustering approaches provide a natural and mathematically grounded means of synthesizing combinatorial (graphical) data and non-combinatorial (vertex or edge meta-data) data into a unified frame of reference, potentially significantly improving the practical performance of these approaches. 

It is worth mentioning one significant challenge associated with this approach; if $M_E$ is not weakly $0$-conformal then the spectrum of the inner product Laplacian is not independent of the choice of signing of the edges (i.e., the choice of the matrix $B$). For example consider the case of the path on three vertices, $v_1, v_2, v_3$ with $\deg(v_2) = 2$, $M_V = I$, $M_E = I + J$ where $J$ is the all-ones matrix.  Viewing the signings as directed orientations of the edges, there are three choices of signings for $B$, corresponding to $v_2$ having in-degree $0$, $1$, or $2$.  As one can transform between the in-degree $0$ and $2$ by reversing all the edges (and consequently all the signs in $B$), there are two distinct inner product Laplacians for this scenario; 
\[ \begin{bmatrix} 2 & -3 & 1\\ -3 & 6 & -3 \\ 1 & -3 & 2  \end{bmatrix} \qquad \textrm{and} \qquad
\begin{bmatrix} 2 & -1 & -1 \\ -1 & 2  & -1 \\ -1  & -1 & 2 \end{bmatrix}.  \]  The spectrum for these two matrices are $(9,1,0)$ and $(3,3,0)$, respectively.  As a consequence, the isoperimetric results of this section should be thought of as a family of isomperimetric inequalities which can be indexed by the orientations of $G$.

Before stating the general form of the isoperimetric inequalities for the inner product Laplacian, it will be helpful to review some standard notation and terms used for isopererimetric properties and extend their definition to the case where the vertices and edges are equipped with inner product spaces, represented by $M_V$ and $M_E$, respectively.  For a collection of vertices, the volume generalizes easily to the inner product context with $\Vol(X) = \one_X^TM_V\one_X$.  It will be helpful to extend this notion to pairs of sets, and so we define $\Vol(X,Y) = \one_X^TM_V\one_Y.$  This approach can be extended to edges between pairs of disjoint sets of vertices $X$ and $Y$, by denoting the set of edges with one endpoint in $X$ and the other endpoint in $Y$, by $E(X,Y)$.  The size of $E(X,Y)$, denoted by $e(X,Y)$, is then given by 
$\one_{E(X,Y)}^TM_E\one_{E(X,Y)}.$  A small complication occurs in the case that  $X \cap Y \neq \varnothing$; the standard approach is to consider  $E(X,Y)$ as a multi-set with the edges in $G[X \cap Y]$ occurring twice and denote by $e(X,Y)$ the cardinality of that multi-set.  However, this approach does not align well with the inner-product-based notion of the size of a set, thus we will view $E(X,Y)$ as a set of edges, independently of the size of $X \cap Y$ and $e(X,Y) = \one_{E(X,Y)}^TM_E\one_{E(X,Y)}$.  We will also denote $E(X,X)$ and $e(X,X)$ by $E(X)$ and $e(X)$, respectively.  Thus, we will denote the standard size of a the set of edges between $X$ and $Y$ as $e(X,Y) + e(X \cap Y)$. Finally, we introduce a new notation, $\Cor(X,Y) = \Vol(X,Y)\Vol(\overline{X},\overline{Y}) - \Vol(X,\overline{Y})\Vol(\overline{X},Y)$ and $\Cor(X) = \Cor(X,X)$, which we call the \emph{correlation} of $X$ and $Y$.  In the case $M_V$ is weakly 0-conformal the correlation of a set $X$ reduces to  $\Vol(X)\Vol(\overline{X})$ as $\Vol(X,\overline{X}) = 0.$

\subsection{The Cheeger Inequality} 
While there are a numerous results relating the spectra of the combinatorial or normalized Laplacian of a graph to structural properties, such as \cite{Alon:spectralColoring,CHAIKEN1978377,kirkoff_orig,wocjan2013new}, perhaps the most widely used relationship is the Cheeger inequality, which connects the spectrum to sparsest cuts.
Depending on how this sparsity is measured, there are different versions of the Cheeger inequality that rely on different graph Laplacians. 
For example, if the sparsity of a cut is measured by the Cheeger ratio, $h(S) = \nicefrac{e(S,\overline{S})}{\min\set{\size{S},\size{\overline{S}}}}$, the second smallest eigenvalue of the combinatorial Laplacian, $\mu_2$ bounds the minimum Cheeger ratio below by $\nicefrac{\mu_2}{2}$ and above by $\sqrt{2\Delta \mu_2}$ where $\Delta$ is the maximum degree of $G$.  If instead, the sparsity of the cut is measured by the conductance, $\Phi(S) = \nicefrac{e(S,\overline{S})}{\min\set{\Vol(S),\Vol(\overline{S})}}$, the second smallest eigenvalue of the normalized Laplacian, $\lambda_2$, bounds the minimum conductance below by $\nicefrac{\lambda_2}{2}$ and above by $\sqrt{2\lambda_2}.$  Given the similarities of these two results (as well as the similarities of the respective proofs) it is unsurprising that there is a common generalization within the context of the inner product Laplacian. 
To this end, we define a generic notation of the sparsity of a cut that incoporates the size information provided by the inner products.

\begin{define}[Inner Product Conductance]
 Let $G = (V,E)$ be a connected graph and let $M_V$ and $M_E$ be inner product spaces defined on the vertices and edges of $G$.  For any proper, non-empty subset $S$ of $V$, the \emph{inner product conductance} of the set $S$ is given by
\[ \Phi(S) = \frac{e(S,\overline{S})}{\min\set{\Vol(S),\Vol(\overline{S})}}. \]  The inner product conductance of the graph, $\Phi$, is the minimum of $\Phi(S)$ over all proper, non-empty subsets $S$ of $V$.
\end{define}

In general, the inner products used to define the inner product conductance will be clear from context. In the case the inner products are not explicitly specified, we will default to the inner products which define the normalized Laplacian, i.e. $M_V$ is the diagonal matrix of degrees and $M_E$ is the identity matrix.  With this notation in hand, we  provide a generalization the Cheeger inequality to the case of inner product Laplacians. 

\begin{theorem}\label{T:cheeger}
Let $G = (V,E)$ be a connected graph and let $B$ be some signed vertex-edge incidence matrix for $G$.  Let $M_V$ and $M_E$ be $\rho_V$- and $\rho_E$-weakly-conformal inner product spaces on the vertices and edges which are $\omega$-compatible.  If $\Phi$ is the inner product conductance of $G$ and $\lambda_2$ is the second smallest eigenvalue for the inner product Laplacian defined by $B$, $M_V$, and $M_E$, then 
\[ \paren{\frac{1-\rho_V}{1+\rho_V}}^7\paren{\frac{1-\rho_E}{1+\rho_E}}^4 \frac{\Phi^2}{2\omega} \leq \lambda_2 \leq \frac{2}{1-\rho_V}\frac{1+\rho_E}{1-\rho_E}\Phi. \]
\end{theorem}

Before turning to the proof of \Cref{T:cheeger}, we note that this statement exactly recovers the Cheeger inequalities for both the normalized Laplacian and combinatorial Laplacian.  In particular, in both cases $M_V$ and $M_E$ are weakly-0-conformal, and so the form of the bound for both the combinatorial and normalized Laplacian is $\frac{\Phi^2}{2\omega} \leq \lambda_2 \leq 2\Phi$.  Noting that for the normalized Laplacian, the inner products are $1$-compatible and $\Phi$ is the standard conductance, yields Cheeger inequality for the normalized Laplacian.  On the other hand, the inner products for the combinatorial Laplacian are $\Delta$-compatible (where $\Delta$ is the maximum degree of $G$) and $\Phi(S)$ is the edge-expansion of the cut $(S,\overline{S})$, i.e. $\Phi(S) =  \frac{\size{E(S,\overline{S})}}{\min{\size{S},\size{\overline{S}}}}$.  Thus, \Cref{T:cheeger} generalizes the Cheeger inequality for both the combinatorial and normalized Laplacian.

\begin{proof}
We first consider the upper bound and recall that by the Rayleigh-Ritz formulation 
\[ \lambda_2 = \min_{f^T Q_V \one = 0} \frac{f^T \IPL f}{f^Tf} = \min_{g^T M_V \one = 0} \frac{g^T B M_E B^T g}{g^T M_V g}.\]
Thus $\lambda_2$ can be upper bounded by the value of the harmonic Rayleigh quotient for any vector $G$ which is orthogonal to $M_V \one$.  Thus, we will fix some $X$ such that $\Vol(X) \leq \Vol(\overline{X})$ and define a vector $g_X$ such that the Rayleigh quotient can be written in terms of $\Phi(X)$.  Choosing such a set $X$ which minimizes $\Phi(X)$ will then yield the desired upper bound.  To that end, fix $X$ and define $g_X = \one_X - \gamma \one$ where $\gamma$ is chosen to be such that $g_X^TM_V\one = 0$, namely $\Vol(G) \gamma = \Vol(X) + \Vol(X,\overline{X}).$

Considering the denominator and numerator of the harmonic Rayleigh quotient separately, we have that 
\[ g^T M_V g = \Vol(X) - 2\gamma \paren{\Vol(X) + \Vol(X,\overline{X})} + \gamma^2 \Vol(G) =  \Vol(X) - \gamma^2 \Vol(G)\]
For the numerator we first note that $B^T\one = 0$ and thus $g_X^T B M_E B^T g_X = \one_X^T B M_E B^T \one_X$.  Further, $B^T \one_X$ is supported precisely on $E(X,\overline{X})$. In particular, if $e = \set{u,v}$ is contained in $G[X]$ then $B_{u,e}$ and $B_{v,e}$ have opposite signs and cancel out, whereas if $e$ is contained in $G[\overline{X}]$ then the entries in $\one_X$ corresponding to $u$ and $v$ are both zero.  We note that, in general,  $B^T\one_X  \neq \one_{E(X,\overline{X})}$.  Indeed, for any $e = \set{u,v} \in E(X,\overline{X})$, the entry $\paren{\one_X^TB}_e = B_{x,e}$ where $x \in \set{u,v} \cap X$.  As a consequence, this entry can be either $1$ or $-1$.  However, as $\abs{B^T\one_X} = \one_{E(X,\overline{X})}$, by applying \Cref{L:conformal} we have that $y^T B M_E B^T y \leq \frac{1+\rho_E}{1-\rho_E} e(X,\overline{X}).$

Combining the numerator and denominator, we have that 
\[ \lambda_2 \leq \frac{1}{\Vol(X) - \gamma^2 \Vol(G)} \paren{\frac{1+\rho_E}{1-\rho_E}} e(X,\overline{X}) =  \frac{\Vol(X)}{\Vol(X) - \gamma^2 \Vol(G)}  \paren{\frac{1+\rho_E}{1-\rho_E}} \Phi(X). \]
Finally, we observe that
\begin{align*}
     \frac{\Vol(X)}{\Vol(X) - \gamma^2 \Vol(G)} &= \frac{\Vol(G)\Vol(X)}{\Vol(G)\Vol(X) - \paren{\Vol(X) + \Vol(X,\overline{X})}^2 } \\
     &= \frac{\Vol(X)^2 + 2 \Vol(X)\Vol(X,\overline{X}) + \Vol(X)\Vol(\overline{X})}{\Vol(X)\Vol(\overline{X}) - \Vol(X,\overline{X})^2} \\
     &\leq \frac{\Vol(X)^2 + 2 \rho_V \Vol(X)\sqrt{\Vol(X)\Vol(\overline{X})} + \Vol(X)\Vol(\overline{X})}{\Vol(X)\Vol(\overline{X})- \rho_V^2\Vol(X)\Vol(\overline{X})} \\
     &= \frac{1}{1-\rho_V^2} \frac{ \Vol(X) + \Vol(\overline{X}) + 2\rho_V \sqrt{\Vol(X)\Vol(\overline{X})}}{\Vol(\overline{X})} \\
     &\leq \frac{2\Vol\paren{\overline{X}} + 2\rho_V\Vol\paren{\overline{X}}}{1-\rho_V^2} \\
     &= \frac{2}{1-\rho_V}.
\end{align*}

We now turn to the lower bound on $\lambda_2$.  To this end, we will follow the probabilistic approach of Trevisian~\cite{trevisian:blog} and heavily use the weak-conformality of $M_V$ and $M_E$ to simplify the proof.  Specifically, we will use the weak conformality of $M_V$ to construct a random variable $\mathcal{S}$ on the space of vertex sets and show that there is a constant $C$ such that $\expect{e(\mathcal{S},\overline{S}) - C \Vol(\mathcal{S})} < 0$.  In turn, this relation (and the choice of $C$) will provide the lower bound.
To that end, for any $f \colon V \rightarrow \R$, define \footnote{Throughout the remainder of the this proof, summands of the form $w_e(f(u) - f(v))^2$ where $e = \set{u,v}$ will be common.  In order to present a more compact notation, we will present this as a sum over edges and with the implicit assumption that $u$ and $v$ are the end points of the edge $e$.} $\mathcal{R}(f) = \frac{\sum_e w_e \paren{f(u)-f(v)}^2}{\sum_v d_v f(v)^2}.$

Let $y \in \R^{\size{V}}$ be an harmonic eigenvector of $\IPL$ associated with $\lambda_2$ and order the vertices $V = \set{v_1, \ldots, v_{\size{V}}}$ so that $y(v_i) \geq y(v_j)$ if $i \leq j$, that is the vertices are ordred by decreasing values of $y$.  Define $k$ as the largest $k$ such that $\sum_{i=1}^{k} d_{v_i} \leq \sum_{j = k+1}^{\size{V}} d_{v_j}$, and define 
\[ y^+(v_i) = \begin{cases} y(v_i) - y(v_k) & i < k \\ 0 & i \geq k\end{cases} \qquad \textrm{and} \qquad y^-(v_j) = \begin{cases} 0 & j \leq k \\ y(v_k) - y(v_j) & j > k \end{cases}. \]  
By construction $y^+$ and $y^-$ have disjoint support and are positive vectors.  
Now, we note that 
\begin{align*}
\frac{\paren{y^+ - y^-}^T B M_E B^T \paren{y^+ - y^-}}{\paren{y^+ - y^-}^TM_V \paren{y^+- y^-}} &= \frac{\paren{y - y(v_k)\one}^T B M_E B^T \paren{y - y(v_k)\one}}{\paren{y - y(v_k)\one}^T M_V \paren{y - y(v_k)\one}} \\
&= \frac{y^TB M_E B^T y}{y^TM_Vy + y(v_k)^2 \Vol(G)} \\
&\leq \lambda_2,
\end{align*}
where the last equality comes from the fact that $y$ is orthogonal to the zero harmonic eigenspace, $M_V\one$.  By using \Cref{L:trace} and the conformality of $M_E$ and $M_V$, we then have that 
\[
\left(\frac{1-\rho_E}{1+\rho_E}\right)\left(\frac{1-\rho_V}{1+\rho_V}\right) \mathcal{R}\paren{y^+ - y^-} \leq \lambda_2.
\]
Since the support of $y^+$ and $y^-$ are disjoint, we observe that $\sum_v \paren{y^+(v) - y^-(v)}^2 d_v = \sum_v \paren{ y^+}^2 d_v + \paren{y^-}^2 d_v$, and additionally, we have 
\begin{align*}
    \sum_{e \in E} w_e\paren{ y^+(u) - y^-(u)  - y^+(v) + y^-(v)}^2 &=  \sum_{e \in E} w_e \paren{y^+(u) - y^+(v)}^2  + w_e\paren{y^-(u) - y^-(v)}^2 \\ 
    &\quad - 2\sum_{e  \in E} \paren{y^+(u) - y^+(v)}\paren{y^-(u) - y^-(v)} \\
    &\geq \sum_{e = \set{u,v}} w_e \paren{y^+(u) - y^+(v)}^2 + w_e\paren{y^-(u) - y^-(v)}^2.
\end{align*} 
The inequality comes from the fact that $\paren{y^+(u) - y^+(v)}\paren{y^-(u) - y^-(v)} \leq 0$.  In particular, it is zero if $\set{u,v}$ is contained the support of $y^+$ or $y^-$, while if ${u,v}$ is not contained in the support of $y^+$ or $y^-$, then  $y^+(u) - y^+(v)$ and $y^-(u) - y^-(v)$ have opposite signs.  As a consequence, letting $y' = \argmin \set{\mathcal{R}(y^+),\mathcal{R}(y^-)}$, we have that $\mathcal{R}(y') \leq \mathcal{R}(y^+ - y^-)$.

By definition, $y'$ is non-negative for all $v$ and further, by rescaling, we may assume without loss of generality that $\max_{v \in V} y'(v) = 1$.  Thus, we may use $y'$ to define a probability space over non-empty subsets of $V$.  Specifically, for any $t \in (0,1)$ define $S_t = \set{v \in V \mid y'(v)^2 \geq t}$ and let $\mathcal{S}$ be the random variable given by $S_t$ where $t$ is chosen uniformly from $(0,1)$.  At this point, our goal is the use the properties of $\mathcal{S}$ to show that there is some $t$ where $\frac{e(S_t,\overline{S_t})}{\Vol(S_t)}$ can be bounded in terms of $\mathcal{R}(y')$.  It is worth noting that this does not directly imply that $\Phi(S_t)$ can be bounded in terms of $\mathcal{R}(y')$ as it may happen that $\Vol(S_t) \geq \Vol(\overline{S_t})$.  We defer dealing with this issue for now and focus instead on the behavior of $e(S_t,\overline{S_t})$ and $\Vol(S_t)$.

We note that if $y'(v)= 0$, then for all $t \in (0,1)$ we have $v \not\in S_t$.  Thus, we may restrict our attention to the vertices in the support of $y'$ and let $k$ be the size of the support.  We also reorder the vertices so that $y'(v_i) \geq y'(v_j)$ for $i \geq j$.  Using this notation, we have that 
\begin{align*}
    \expect{\Vol(\mathcal{S})} &= \sum_{i = 1}^{k} \paren{y'(v_i)^2 - y'(v_{i+1})^2} \Vol(S_{y'(v_i)}) \\
    &\geq \frac{1-\rho_V}{1+\rho_V}\sum_{i = 1}^{k} \paren{y'(v_i)^2 - y'(v_{i+1})^2} \sum_{j = 1}^i d_{v_j} \\
    &= \frac{1-\rho_V}{1+\rho_V} \sum_{j = 1}^k d_{v_j} y'(v_j)^2.  \\
\end{align*}
Furthermore, we have 
\begin{align*}
    \expect{e(\mathcal{S},\overline{\mathcal{S}})} &= \sum_{i=1}^{k} \paren{y'(v_i)^2 - y'(v_{i+1})^2} e(S_{y'(v_i)},\overline{S_{y'(v_i)}})\\
    &\leq \frac{1+\rho_E}{1-\rho_E} \sum_{i=1}^{k} \paren{y'(v_i)^2 - y'(v_{i+1})^2} \sum_{u \in S_{y'(v_i)}} \sum_{v \not\in S_{y'(v_i)}} w_e \one_{\set{u,v} \in E} \\
    &= \frac{1+\rho_E}{1-\rho_E} \sum_{e \in E} w_e\abs{y'(u)^2 - y'(v^2} \\
    &\leq \frac{1+\rho_E}{1-\rho_E} \sqrt{ \sum_{e \in E} w_e \paren{y'(u) - y'(v)}^2} \sqrt{ \sum_{e \in E} w_e \paren{y'(u) + y'(v)}^2} \\
    &\leq \frac{1+\rho_E}{1-\rho_E} \sqrt{ \sum_{e \in E} w_e \paren{y'(u) - y'(v)}^2} \sqrt{ \sum_v \paren{2 y'(v)^2 \sum_{e \in E(v)} w_e} }\\ 
    &\leq \frac{1+\rho_E}{1-\rho_E} \sqrt{ \sum_{e \in E} w_e \paren{y'(u) - y'(v)}^2} \sqrt{ \sum_v 2 y'(v)^2  \left(\frac{1+\rho_E}{1-\rho_E}\right) e(v)}\\ 
    &\leq \paren{\frac{1+\rho_E}{1-\rho_E}}^{\nicefrac{3}{2}}  \sqrt{ \sum_{\set{u,v}=e \in E} w_e \paren{y'(u) - y'(v)}^2} \sqrt{ \sum_v 2 y'(v)^2 \omega d_v }\\ 
\end{align*}
Thus we may conclude that 
\[ \expect{ e(\mathcal{S},\overline{\mathcal{S}}) - \paren{\frac{1+\rho_V}{1-\rho_V}}\paren{\frac{1 + \rho_E}{1-\rho_E}}^{\nicefrac{3}{2}}\sqrt{2\omega \mathcal{R}(y')} \Vol(\mathcal{S})} \leq 0. \] As a consequence, there is some $t \in (0,1)$ such that 
\[ \frac{e(S_t,\overline{S_t})}{\Vol(S_t)} \leq \paren{\frac{1+\rho_V}{1-\rho_V}}\paren{\frac{1 + \rho_E}{1-\rho_E}}^{\nicefrac{3}{2}}\sqrt{2\omega \mathcal{R}(y')} \leq  \paren{\frac{1+\rho_V}{1-\rho_V}}\paren{\frac{1 + \rho_E}{1-\rho_E}}^{\nicefrac{3}{2}}\sqrt{2\omega \left(\frac{1+\rho_E}{1-\rho_E}\right) \left(\frac{1+\rho_V}{1-\rho_V}\right)\lambda_2}.\] 

In order to complete the proof, it is necessary to handle the case where $\Vol(S_t) > \Vol(\overline{S_t})$. To this end, we observe that by construction of $y^+$ and $y^-$, we have that $\sum_{v \in S_t} d_v \leq \frac{1}{2}\trace(M_V)$ and $\sum_{v \not\in S_t} d_v \geq \frac{1}{2}\trace(M_V)$.  In particular, by \Cref{L:trace}, this implies that $\Vol(S_t) \leq \frac{1+\rho_V}{1-\rho_V} \frac{1}{2}\trace(M_V)$ and $\Vol(\overline{S_t}) \geq \frac{1-\rho_V}{1+\rho_V} \frac{1}{2}\trace(M_V).$  We may conclude that 
\[ \frac{e(S_t,\overline{S_t})}{\Vol(\overline{S_t})} \leq \paren{\frac{1+\rho_V}{1-\rho_V}}^2 \frac{e(S_t,\overline{S_t})}{\Vol(S_t)}\]
and thus 
\[ \Phi \leq \Phi(S_t) \leq \paren{\frac{1+\rho_V}{1-\rho_V}}^{\nicefrac{7}{2}}\paren{\frac{1+\rho_E}{1-\rho_E}}^2 \sqrt{2\omega \lambda_2},\]
as desired.
\end{proof}

We note that \Cref{T:cheeger} combined with the proof of \Cref{T:neumann} immediately yields a Cheeger-upper bound on the Neumann eigenvalue.
\begin{corollary}
    Let $G$ be a connected graph and let $S$ be a proper subset of the vertices with size at least 2.  Define the \emph{$S$-local conductance} as \[ \Phi_S = \min_{T \subseteq S} \frac{e(T,\overline{T})}{\min\set{\Vol(T),\Vol(S-T)}}.\]
    If $\lambda_S$ is the Neumann eigenvalue for the set $S$, then  $\lambda_S \leq 2\Phi_S$.   
\end{corollary}
As the compatibility of the inner products defined in the proof of \Cref{T:neumann} is $\nicefrac{1}{\epsilon}$, the limit of the lower-bound in \Cref{T:cheeger} is 0. It would be interesting to resolve whether this lower bound could be improved by a clever choice of inner products or refinements in the proof of \Cref{T:cheeger}.
 Indeed, having lower and upper bounds on $\lambda_S$ in terms of a polynomial of $\Phi_S$ would provide a new approach towards refuting the small set expansion hypothesis\footnote{The small set expansion hypothesis holds that it is \NPc\ to distinguish between a $n$-vertex, $d$-regular graph such that $e(S,\overline{S}) \geq (1-\epsilon)d\size{S}$ for all sets of size at most $\nicefrac{n}{\log_2(n)}$ and an $n$-vertex, $d$-regular graph which contains a set $S$ of size at most $\nicefrac{n}{\log_2(n)}$ such that $e(S,\overline{S}) \leq \epsilon d\size{S}$.} by providing efficient means of testing the edge-expansion of multiple small sets simultaneously. As the small set expansion hypothesis implies the Unique Games Conjecture~\cite{UGC}, and is equivalent to a restricted version of the Unique Games Conjecture~\cite{restrictedUGC}, we suspect that no such improvement is possible. 

\subsection{The Expander Mixing Lemma}
Over the years there have been a variety of different ``expander mixing lemma"-type results which bound the size of the boundary between two sets  $X$ and $Y$ in terms of the spectrum of an associated matrix and some measure of the size of $X$ and $Y$.  We forgo a thorough recounting of the history of such results and instead highlight three different versions of the expander mixing lemma\footnote{The expander mixing lemma is also occasionally, and perhaps more appropriately, called the \emph{discrepancy inequality} because of its close connection to  edge discrepancies and pseudorandomness~\cite{chung1997spectral}. However, as the term discrepancy appears in numerous mathematical contexts we will continue to use the more discoverable term ``expander mixing lemma."} which are representative of different aspects and similarities of the various expander mixing lemmas in the literature.

The first version of the expander mixing lemma we highlight is for the normalized Laplacian and is perhaps the most reminiscent of the early results~\cite{Alon:EML,Bussmaker:EML,haemers1979eigenvalue,Krivelevich:EML} for $d$-regular graphs.
\begin{nEML}
    Let $G$ be an $n$-vertex connected graph and let $0 = \lambda_1 < \lambda_2 \leq \cdots \lambda_n$ be the eigenvalues of the normalized Laplacian.  Let $\lambda = \max_{i \neq 1} \abs{1-\lambda_i}$.  For any two sets $X$ and $Y$, 
    \[\abs{ e(X,Y) + e(X \cap Y) - \frac{\Vol(X)\Vol(Y)}{\Vol(G)}} \leq \lambda \frac{\sqrt{\Vol(X)\Vol(\overline{X})\Vol(Y)\Vol(\overline{Y})}}{\Vol(G)}.\]
\end{nEML}

Informally, this result is often expressed as stating that the number of edges between $X$ and $Y$ is approximately equal to the number of edges that would be expected in a random graph with the same volumes for $X$ and $Y$.  However, a careful examination of the proof reveals that if $X$ and $Y$ are disjoint, the range over which $e(X,Y)$ can vary can be tightened.   Specifically, when $X$ and $Y$ are disjoint, $\lambda$ can be replaced by $\frac{\lambda_n-\lambda_1}{2}$ by shifting the spectrum by $\frac{\lambda_n+\lambda_1}{2}$ instead of $1$.  Applying this alteration to the case where $X$ and $Y$ are not disjoint leads to a similar result, differing only in a correction factor based on the size of the intersection for $X$ and $Y$. This observation was originally made by Chung for the combinatorial Laplacian~\cite{chung2004discrete}.\footnote{We note that there is a typographical error in \cite[Theorem 7]{chung2004discrete} in that the division by $n$ on the right hand side is missing.  However, as the proof of this theorem is a straightforward modification of \cite[Theorem 4]{chung2004discrete}, which has the division by $n$, it is easy confirm that the provided statement is correct.}
\begin{cEML}
    Let $G$ be an $n$-vertex connected graph and let $0 = \sigma_0 < \sigma_1 \leq \cdots \sigma_n$ be the eigenvalues of the combinatorial Laplacian.  For any two sets $X$ and $Y$, 
    \[ \abs{ e(X,Y) + e(X \cap Y) - \Vol(X \cap Y) -  \frac{\sigma_n + \sigma_1}{2}\paren{\frac{\size{X}\size{Y}}{n} - \size{X \cap Y}}} \leq \frac{\sigma_n - \sigma_1}{2} \frac{\sqrt{\size{X}(n-\size{X})\size{Y}(n-\size{Y})}}{n}. \]
\end{cEML}

It is worth noting, that unlike the the normalized Laplacian expander mixing lemma, the combinatorial Laplacian variant incorporates terms which are not related to the associated inner product spaces.  Specifically, $\Vol(X \cap Y)$ is not the ``size" of the set of vertices $X \cap Y$ in vertex inner product space for the combinatorial Laplacian. 
 Instead, it should be viewed as the size of the incident edges to $X \cap Y$, that is, $\Vol(X \cap Y) = \sum_{a \in X \cap Y} e(\set{a}),$ which is a representable in terms of the edge inner product.  
 
 More recently, Bollobas and Nikiforov~\cite{Bollobas:EML,nikiforov2009cut}  recognized that similar proof techniques can be adapted to the case of non-negative, not necessarily square, matrices.  This results in matrix expander mixing lemmas along the line of the following result of Butler~\cite{Butler:EML}.
\begin{mEML}
    Let $B$ be an $m \times n$-matrix with nonnegative entries and no zero columns or rows.  Let $R$ and $C$ be the unique diagonal matrices such that $B\one = R\one$ and $\one^TB = \one^TC$.  For all sets $S \subseteq [m]$ and $T \subseteq [n]$, 
    \[ \abs{ \one_S^TB\one_T - \frac{\one_S^TB\one \one^T B \one_T}{\one^TB\one}} \leq \sigma_2\paren{R^{-\half}BC^{-\half}} \sqrt{\one_S^TB\one \one^TB\one_T},\] where $\sigma_2(\cdot)$ takes a matrix to the second largest singular value.
\end{mEML}

Before generalizing the expander mixing lemma for the inner product Laplacian, we consider a small example which is helpful for understanding the trade-offs necessary in this context.  To begin, define the graph $G = (V,E)$ as the union of three sets, $A$, $B$, and $C$, with $\size{A} = \size{B} = k$ and $\size{C} = 2k$.  The edge set is the union of a complete bipartite graph between $A$ and $B$, a $k$-regular bipartite graph between $A \cup B$ and $C$, and a perfect matching on $C$.  The vertex and edge inner products are defined according to this partition of vertices and edges by the matrices
\[ M_V =
\kbordermatrix{
& A & B & C \\
A & (2k^2 + 2k)I & 0 & 0 \\
B & 0 & (2k^2 + 2k)I & 0 \\
C & 0 & 0 & (2k^2+k)I } \qquad M_E = \kbordermatrix{
& E(A,B) & E(A\cup B,C) & E(C) \\
E(A,B) & I + 2 J & 0 & 0 \\
E(A\cup B,C) & 0 & I & 0 \\
E(C) & 0 & 0 & 2k^2 I}
\]
As $M_V$ is diagonal with strictly positive diagonal, it is positive definite and weakly $0$-conformal.  Because of the block diagonal structure of $M_E$ it is easy to verify that $M_E$ is also positive definite.  Further, it is relatively straightforward to compute the weak conformality of $M_E$ by noting that since the non-zero diagonal terms are all in the principle block formed by $E(A,B)$, it suffices to consider only the witnessing partitions which partition $E(A,B)$.  From there, a straightforward optimization yields that the weak-conformality of $M_E$ is $\frac{k^2}{k^2+1}.$

We further note that $M_V$ and $M_E$ are perfectly $1$-compatible with respect to $G$ and thus the spectrum of $\IPL$ is contained in $[0,2]$.  Thus, if an expander mixing lemma type-result along the form of the results described above were to hold, we would expect that $e(A,B)$ would grow (in terms of $k$) at the same rate as one of 
\[ \frac{\Vol(A)\Vol(B)}{\Vol(G)}, \qquad \frac{\Vol(A,V)\Vol(B,V)}{\Vol(G)}, \textrm{or} \qquad \frac{\sqrt{\Vol(A)\Vol(\overline{A})\Vol(B)\Vol(\overline{B})}}{\Vol(G)}.\]
To see this is not the case, we observe that 
\begin{align*}
    e(A,B) &= 2k^4 + k^2, \\
    \Vol(A), \Vol(B), \Vol(A,V), \Vol(B,V) &= 2k^3 + 2 k^2, \\
    \Vol(\overline{A}), \Vol(\overline{B}) &= 6k^3 + 4 k^2, \quad \textrm{and} \\
    \Vol(G) &= 8k^3+6k^2. \\
    \end{align*}
In particular, $e(A,B) \in \bigTheta{k^4}$, while the candidate terms for the bounds are all $\bigTheta{k^3}$. In order to resolve this issue, it is necessary to explicitly incorporate the weak-conformality of $M_V$ and $M_E$ into the expander mixing lemma bounds.  While this shows a direct generalization of the expander mixing lemma is not possible for the inner product Laplacian, by incorporating terms based on the weak conformality $M_E$ into the bound we can produce the following general result. 

\begin{theorem}\label{T:EML} 
Let $G = (V,E)$ be a connected graph equipped with inner product spaces on the vertices and edges, $M_V = Q_V^2$ and $M_E = Q_E^2$, respectively.  Let $0 = \lambda_1 < \lambda_2 \cdots \leq \lambda_n$ be the spectrum of the associated inner product Laplacian.  If $G$ is connected, then for any two sets $X, Y \subseteq V$,
\[ \abs{e(X,Y) + e(X\cap Y) -\sum_{a \in X \cap Y} e(\set{a},\overline{\set{a}}) + \frac{\lambda_n+\lambda_2}{2} \frac{\Cor(X,Y)}{\Vol(G)}}\] is bounded above by \[\frac{\lambda_n-\lambda_2}{2} \frac{\sqrt{\Cor(X)\Cor(Y)}}{\Vol(G)} +   \frac{12\rho_E}{1-\rho_E^2} \trace(M_E). \] 
\end{theorem}
\begin{proof}
Let $\set{\vartheta_i}$ be an orthonormal basis for the eigenspaces of $\IPL$ where $\lambda_i$ is the eigenvalue associated with $\vartheta_i$.  Following a standard proof technique for the expander mixing lemma, for $\tau = \frac{\lambda_n+\lambda_1}{2}$ we define $\mathcal{L}_{\tau} = \tau I - \IPL - \tau \vartheta_1\vartheta_1^T$,  $\rho_X = Q_V \one_X$, and $\rho_Y = Q_V \one_Y$ and consider $\rho_X^T\mathcal{L}_{\tau}\rho_Y$.  We note that by construction, $\set{\vartheta_i}$ is also an orthonormal basis for the eigenspaces of $\mathcal{L}_{\tau}$.  Indeed, the spectrum of $\mathcal{L}_{\tau}$ can be recovered by the simple transformation $\lambda \mapsto \tau - \lambda$ of the non-zero eigenvalues of $\IPL.$  Thus we may rewrite $\rho_X = \sum_i \alpha_i \vartheta_i$ and $\rho_Y = \sum_i \beta_i \vartheta_i$, and thus
\[\abs{\rho_X^T \mathcal{L}_{\omega} \rho_Y} = \abs{ \sum_{i \geq 2} (\tau-\lambda_i)\alpha_i \beta_i } \leq \frac{\lambda_n - \lambda_2}{2} \sqrt{\sum_{i \geq 2} \alpha_i^2}\sqrt{\sum_{i \geq 2} \beta_i^2} = \frac{\lambda_n - \lambda_2}{2} \sqrt{\norm{\rho_X}^2 - \alpha_1^2}\sqrt{\norm{\rho_Y}^2 - \beta_1^2}.\]
Now, we note that
\begin{align*}
     \norm{\rho_X}^2 - \alpha_i^2 &= \one_X^TM_V\one_X - \paren{ \frac{\one_X^TQ_V^T Q_V \one}{\sqrt{\one^T Q_V^TQ_V \one}}}^2 \\
     &= \Vol(X) - \frac{\paren{\Vol(X) + \Vol(X,\overline{X})}^2}{\Vol(G)} \\
     &= \frac{1}{\Vol(G)}\paren{ \Vol(X)\Vol(G) - \Vol(X)^2 - 2\Vol(X)\Vol(X,\overline{X}) - \Vol(X,\overline{X})^2} \\
     &= \frac{1}{\Vol(G)}\paren{\Vol(X)\Vol(\overline{X}) - \Vol(X,\overline{X})^2} \\
     &= \frac{\Cor{X}}{\Vol(G)}
\end{align*}
Hence, we have that
\[ \abs{\rho_X^T\mathcal{L}_{\omega}\rho_Y} \leq \frac{\lambda_n - \lambda_1}{2}\frac{\sqrt{ \Cor(X)\Cor(Y)} }{\Vol(G)}. \]

Thus, to complete the proof it suffices to bound 
\[\abs{ e(X,Y) + e(X\cap Y)  - \sum_{a \in X \cap Y} e(\set{a},\overline{\set{a}}) + \tau \frac{\Cor(X,Y)}{\Vol(G)} - \rho_X^T\mathcal{L}_{\tau}\rho_Y}.\]
To this end, we first note that 
\begin{align*}
    \rho_X^T\rho_Y - \rho^T_X\vartheta_1\vartheta_1^T\rho_Y &= \Vol(X,Y) - \frac{\Vol(X,V)}{\sqrt{\Vol{G}}} \frac{\Vol(V,Y)}{\sqrt{\Vol{G}}} \\
    &= \frac{1}{\Vol(G)}\left[\Vol(X,Y) \Vol(G) - \Vol(X,V)\Vol(V,Y)\right] \\
    &= \frac{1}{\Vol(G)}\left[ \Vol(X,Y) \left(\Vol(X,Y) + \Vol(X,\overline{Y}) + \Vol(\overline{X},Y) + \Vol(\overline{X},\overline{Y}) \right) \right. \\
    &\phantom{=\frac{1}{\Vol(G)}}\quad - \left. \paren{\Vol(X,Y)+\Vol(X,\overline{Y})}\paren{\Vol(X,Y) + \Vol(\overline{X},Y)} \right] \\
    &= \frac{1}{\Vol(G)}\left[\Vol(X,Y)\Vol(\overline{X},\overline{Y}) - \Vol(X,\overline{Y)}\Vol(\overline{X},Y)\right]\\
    &= \frac{\Cor(X,Y)}{\Vol(G)}.
\end{align*}

As a consequence,
\[\abs{ e(X,Y) + e(X\cap Y)  - \sum_{a \in X \cap Y} e(\set{a},\overline{\set{a}}) + \tau \frac{\Cor(X,Y)}{\Vol(G)} - \rho_X^T\mathcal{L}_{\tau}\rho_Y}\]
reduces to 
\[\abs{ e(X,Y) + e(X\cap Y)  - \sum_{a \in X \cap Y} e(\set{a},\overline{\set{a}}) + \rho_X^T\IPL\rho_Y}.\]

In order to bound this term, we will break it into two terms which may at first seem to be unmotivated, but will arise naturally from the analysis of the individual terms.  
Before delving into the the decomposition of this term, it will be helpful to introduce some additional notation.  First, we note that $X$ and $Y$ induce a partition of the vertices into four sets,  $(X',A,Y',W)$, where $X' = X - Y$, $Y' = Y - X$, $A = X \cap Y$ and $W = \overline{X \cup Y} = V - X' - A - Y'$.  In what follows, we will often be dealing with the edges between pairs of these sets, so we introduce a compact nation, $B_{S,T}$, for the vector of signed edges between two disjoint sets, $S$ and $T$.  The vector $B_{S,T}$ takes on the value of $\set{-1,1}$ on edge between $S$ and $T$ with the sign determined by $B$.  That is,  $B_{S,T}$ is $B^T\one_S$ restricted to the edge set $E(S,T)$.  It is important to note that the order of the subscripts matter as $B_{S,T} = -B_{T,S}$. With this notation in hand, we now  proceed to bound 
\[ \abs{ \rho_X^T\IPL\rho_Y + \norm[E]{B_{X',Y'}}^2 - \norm[E]{B_{A,W}}^2} \quad \textrm{and} \quad \abs{ e(X,Y) + e(A)  - \sum_{a \in A} e(\set{a},\overline{\set{a}}) - \norm[E]{B_{X',Y'}}^2 + \norm[E]{B_{A,W}}^2} \]
separately.  For the first term, we note that 
 $\rho_X^T \IPL \rho_Y = \one_X^T Q_V^{\half}Q_V^{-\half} B M_E B^T Q_V^{-\half}W_V^{\half} \one_Y = \one_X^TBM_EB^T\one_Y$ can be rewritten as
\[ \paren{B_{X',Y'} + B_{X',W} + B_{A,Y'} + B_{A,W}}^T M_E \paren{B_{Y',X'} + B_{Y',W} + B_{A,X'} + B_{A,W}}.\]
This expands to $-\norm[E]{B_{X',Y'}}^2 + \norm[E]{B_{A,W}}^2$ plus a correction term.  Observing that  $B_{X',Y'}^TM_EB_{A,W} + B_{A,W}^TM_EB_{Y',X'} = 0$, this correction term is corresponds to the subblocks of $M_E$ indicated in \Cref{F:blocks}.  By applying the weak-conformality of $M_E$ and the arithmetic-geometric mean inequality, 
the magnitude of the correction can be bounded by
\[ 2\rho_E\paren{ \norm[E]{B_{X',Y'}}^2 +\norm[E]{B_{X',A}}^2 +\norm[E]{B_{X',W}}^2 +\norm[E]{B_{Y',A}}^2 +\norm[E]{B_{Y',W}}^2 +\norm[E]{B_{A,W}}^2 }.\]
Finally, by applying \Cref{L:trace} and recalling that $w_f = \one_f^TM_E\one_f$, this can be bounded above by 
\[ \frac{2\rho_E(1+\rho_E)}{1-\rho_E} \paren{ \sum_{f \in E(W,\overline{W})} w_f + \sum_{f \in E(\overline{W})} w_f}.\] 
\begin{figure}
    \centering
    \begin{tikzpicture}[
    empty/.style = {draw=black,very thin,fill=white},
    full/.style ={draw=black,very thin,fill=blue}
    ]
        \def\sz{20pt}
        \draw[empty] (0pt,0pt) rectangle ++(\sz,\sz);
        \draw[empty] (0pt,\sz) rectangle ++(\sz,\sz);
        \draw[full] (0pt,2*\sz) rectangle ++(\sz,\sz);
        \draw[full] (0pt,3*\sz) rectangle ++(\sz,\sz);
        \draw[empty] (0pt,4*\sz) rectangle ++(\sz,\sz);
        \draw[empty] (0pt,5*\sz) rectangle ++(\sz,\sz);

        \draw[full] (\sz,0pt) rectangle ++(\sz,\sz);
        \draw[empty] (\sz,\sz) rectangle ++(\sz,\sz);
        \draw[full] (\sz,2*\sz) rectangle ++(\sz,\sz);
        \draw[full] (\sz,3*\sz) rectangle ++(\sz,\sz);
        \draw[empty] (\sz,4*\sz) rectangle ++(\sz,\sz);
        \draw[full] (\sz,5*\sz) rectangle ++(\sz,\sz);

        \draw[empty] (2*\sz,0pt) rectangle ++(\sz,\sz);
        \draw[empty] (2*\sz,\sz) rectangle ++(\sz,\sz);
        \draw[empty] (2*\sz,2*\sz) rectangle ++(\sz,\sz);
        \draw[empty] (2*\sz,3*\sz) rectangle ++(\sz,\sz);
        \draw[empty] (2*\sz,4*\sz) rectangle ++(\sz,\sz);
        \draw[empty] (2*\sz,5*\sz) rectangle ++(\sz,\sz);

        \draw[empty] (3*\sz,0pt) rectangle ++(\sz,\sz);
        \draw[empty] (3*\sz,\sz) rectangle ++(\sz,\sz);
        \draw[empty] (3*\sz,2*\sz) rectangle ++(\sz,\sz);
        \draw[empty] (3*\sz,3*\sz) rectangle ++(\sz,\sz);
        \draw[empty] (3*\sz,4*\sz) rectangle ++(\sz,\sz);
        \draw[empty] (3*\sz,5*\sz) rectangle ++(\sz,\sz);

        \draw[full] (4*\sz,0pt) rectangle ++(\sz,\sz);
        \draw[empty] (4*\sz,\sz) rectangle ++(\sz,\sz);
        \draw[full] (4*\sz,2*\sz) rectangle ++(\sz,\sz);
        \draw[full] (4*\sz,3*\sz) rectangle ++(\sz,\sz);
        \draw[empty] (4*\sz,4*\sz) rectangle ++(\sz,\sz);
        \draw[full] (4*\sz,5*\sz) rectangle ++(\sz,\sz);

        \draw[empty] (5*\sz,0pt) rectangle ++(\sz,\sz);
        \draw[empty] (5*\sz,\sz) rectangle ++(\sz,\sz);
        \draw[full] (5*\sz,2*\sz) rectangle ++(\sz,\sz);
        \draw[full] (5*\sz,3*\sz) rectangle ++(\sz,\sz);
        \draw[empty] (5*\sz,4*\sz) rectangle ++(\sz,\sz);
        \draw[empty] (5*\sz,5*\sz) rectangle ++(\sz,\sz);

        \node at (-1.25*\sz,.5*\sz) () {$E(A,W)$};
        \node at (-1.25*\sz,1.5*\sz) () {$E(Y',W)$};
        \node at (-1.25*\sz,2.5*\sz) () {$E(Y',A)$};
        \node at (-1.25*\sz,3.5*\sz) () {$E(X',W)$};
        \node at (-1.25*\sz,4.5*\sz) () {$E(X',A)$};
        \node at (-1.25*\sz,5.5*\sz) () {$E(X',Y')$};
        \node[rotate=90] at (.5*\sz,7.25*\sz) () {$E(X',Y')$};
        \node[rotate=90] at (1.5*\sz,7.25*\sz) () {$E(X',A)$};
        \node[rotate=90] at (2.5*\sz,7.25*\sz) () {$E(X',W)$};
        \node[rotate=90] at (3.5*\sz,7.25*\sz) () {$E(Y',A)$};
        \node[rotate=90] at (4.5*\sz,7.25*\sz) () {$E(Y',W)$};
        \node[rotate=90] at (5.5*\sz,7.25*\sz) () {$E(A,W)$};
    \end{tikzpicture}
    \caption{Illustration of the subblocks of $M_E$ participating in the correction term.}
    \label{F:blocks}
\end{figure}

For the second term, we apply \Cref{L:trace} to each term individually and have the bound 
\[ \frac{4\rho_E}{1-\rho_E^2}\paren{\sum_{f \in E(W,\overline{W})} w_f + \sum_{f \in E(\overline{W})} w_f + 2\sum_{f \in E(A)} w_f}.\]
As $2\rho_E(1+\rho_E)^2 \leq 12\rho_E$, the combined error terms can be bounded above by $\frac{12\rho_E}{1-\rho_E^2}\trace(M_E),$ as desired.
\end{proof}

\bibliographystyle{siam}
\bibliography{references}

\begin{thebibliography}{10}

\bibitem{agarwal2006higher}
{\sc Sameer Agarwal, Kristin Branson, and Serge Belongie}, {\em Higher order
  learning with graphs}, in Proceedings of the 23rd international conference on
  Machine learning, 2006, pp.~17--24.

\bibitem{agarwal2005beyond}
{\sc Sameer Agarwal, Jongwoo Lim, Lihi Zelnik-Manor, Pietro Perona, David
  Kriegman, and Serge Belongie}, {\em Beyond pairwise clustering}, in 2005 IEEE
  Computer Society Conference on Computer Vision and Pattern Recognition
  (CVPR'05), vol.~2, IEEE, 2005, pp.~838--845.

\bibitem{aksoy2024scalable}
{\sc Sinan~G Aksoy, Ilya Amburg, and Stephen~J Young}, {\em Scalable tensor
  methods for nonuniform hypergraphs}, SIAM Journal on Mathematics of Data
  Science, 6 (2024), pp.~481--503.

\bibitem{aksoy2024unifying}
{\sc Sinan~G Aksoy, Bo~Fang, Roberto Gioiosa, William~W Kay, Hyungro Lee,
  Jenna~A Bilbrey, Madelyn~R Shapiro, and Stephen~J Young}, {\em Unifying
  combinatorial and graphical methods in artificial intelligence}, tech.
  report, Pacific Northwest National Laboratory (PNNL), Richland, WA (United
  States), 2024.

\bibitem{Alon:EML}
{\sc N.~Alon and F.R.K. Chung}, {\em Explicit construction of linear sized
  tolerant networks}, Discrete Mathematics, 72 (1988), pp.~15--19.

\bibitem{Alon:spectralColoring}
{\sc Noga Alon and Nabil Kahale}, {\em A spectral technique for coloring random
  3-colorable graphs}, SIAM Journal on Computing, 26 (1997), pp.~1733--1748.

\bibitem{banerjee2021spectrum}
{\sc Anirban Banerjee}, {\em On the spectrum of hypergraphs}, Linear algebra
  and its applications, 614 (2021), pp.~82--110.

\bibitem{benko2024hypermagnet}
{\sc Tatyana Benko, Martin Buck, Ilya Amburg, Stephen~J Young, and Sinan~G
  Aksoy}, {\em Hypermagnet: A magnetic laplacian based hypergraph neural
  network}, arXiv preprint arXiv:2402.09676,  (2024).

\bibitem{Bolla:Hypergraph}
{\sc Marianna Bolla}, {\em Spectra, euclidean representations and clusterings
  of hypergraphs}, Discrete Mathematics, 117 (1993), pp.~19--39.

\bibitem{Bollobas:EML}
{\sc Béla Bollobás and Vladimir Nikiforov}, {\em Hermitian matrices and
  graphs: singular values and discrepancy}, Discrete Mathematics, 285 (2004),
  pp.~17--32.

\bibitem{brouwer2011spectra}
{\sc Andries~E Brouwer and Willem~H Haemers}, {\em Spectra of graphs}, Springer
  Science \& Business Media, 2011.

\bibitem{Bussmaker:EML}
{\sc F.C. BUSSEMAKER, D.M. CVETKOVIC, and J.J. SEIDEL}, {\em Graphs related to
  exceptional root systems}, in Geometry and Combinatorics, D.G. Corneil and
  R.~Mathon, eds., Academic Press, 1991, pp.~94--100.

\bibitem{Butler:EML}
{\sc Steve Butler}, {\em Using discrepancy to control singular values for
  nonnegative matrices}, Linear Algebra and its Applications, 419 (2006),
  pp.~486--493.

\bibitem{butler2010note}
{\sc Steve Butler}, {\em A note about cospectral graphs for the adjacency and
  normalized laplacian matrices}, Linear and Multilinear Algebra, 58 (2010),
  pp.~387--390.

\bibitem{Cardoso:hypergraph}
{\sc Kauê Cardoso and Vilmar Trevisan}, {\em The signless laplacian matrix of
  hypergraphs}, Special Matrices, 10 (2022), pp.~327--342.

\bibitem{Carletti:RWhypergraph}
{\sc Timoteo Carletti, Federico Battiston, Giulia Cencetti, and Duccio
  Fanelli}, {\em Random walks on hypergraphs}, Phys. Rev. E, 101 (2020),
  p.~022308.

\bibitem{CHAIKEN1978377}
{\sc S~Chaiken and D.J Kleitman}, {\em Matrix tree theorems}, Journal of
  Combinatorial Theory, Series A, 24 (1978), pp.~377--381.

\bibitem{chitra2019random}
{\sc Uthsav Chitra and Benjamin Raphael}, {\em Random walks on hypergraphs with
  edge-dependent vertex weights}, in International conference on machine
  learning, PMLR, 2019, pp.~1172--1181.

\bibitem{chung2004discrete}
{\sc Fan Chung}, {\em Discrete isoperimetric inequalities}, Surveys in
  differential geometry, 9 (2004), pp.~53--82.

\bibitem{chung2005laplacians}
\leavevmode\vrule height 2pt depth -1.6pt width 23pt, {\em Laplacians and the
  cheeger inequality for directed graphs}, Annals of Combinatorics, 9 (2005),
  pp.~1--19.

\bibitem{chung2006diameter}
\leavevmode\vrule height 2pt depth -1.6pt width 23pt, {\em The diameter and
  laplacian eigenvalues of directed graphs}, the electronic journal of
  combinatorics,  (2006), pp.~N4--N4.

\bibitem{chung1992laplacian}
{\sc Fan~RK Chung}, {\em The laplacian of a hypergraph.}, Expanding graphs, 10
  (1992), pp.~21--36.

\bibitem{chung1997spectral}
\leavevmode\vrule height 2pt depth -1.6pt width 23pt, {\em Spectral graph
  theory}, vol.~92, American Mathematical Soc., 1997.

\bibitem{chung1996combinatorial}
{\sc Fan~RK Chung and Robert~P Langlands}, {\em A combinatorial laplacian with
  vertex weights}, journal of combinatorial theory, Series A, 75 (1996),
  pp.~316--327.

\bibitem{cohen2016faster}
{\sc Michael~B Cohen, Jonathan Kelner, John Peebles, Richard Peng, Aaron
  Sidford, and Adrian Vladu}, {\em Faster algorithms for computing the
  stationary distribution, simulating random walks, and more}, in 2016 IEEE
  57th annual symposium on foundations of computer science (FOCS), IEEE, 2016,
  pp.~583--592.

\bibitem{cooper2020adjacency}
{\sc Joshua Cooper}, {\em Adjacency spectra of random and complete
  hypergraphs}, Linear Algebra and its Applications, 596 (2020), pp.~184--202.

\bibitem{cooper2012spectra}
{\sc Joshua Cooper and Aaron Dutle}, {\em Spectra of uniform hypergraphs},
  Linear Algebra and its applications, 436 (2012), pp.~3268--3292.

\bibitem{cvetkovic2005signless}
{\sc Drago{\v{s}} Cvetkovi{\'c}}, {\em Signless laplacians and line graphs},
  Bulletin (Acad{\'e}mie serbe des sciences et des arts. Classe des sciences
  math{\'e}matiques et naturelles. Sciences math{\'e}matiques),  (2005),
  pp.~85--92.

\bibitem{haemers1979eigenvalue}
{\sc Wilhelmus~Hubertus Haemers}, {\em Eigenvalue techniques in design and
  graph theory}, PhD thesis, 1979.

\bibitem{haemers2004enumeration}
{\sc Willem~H Haemers and Edward Spence}, {\em Enumeration of cospectral
  graphs}, European Journal of Combinatorics, 25 (2004), pp.~199--211.

\bibitem{hayashi2020hypergraph}
{\sc Koby Hayashi, Sinan~G Aksoy, Cheong~Hee Park, and Haesun Park}, {\em
  Hypergraph random walks, laplacians, and clustering}, in Proceedings of the
  29th acm international conference on information \& knowledge management,
  2020, pp.~495--504.

\bibitem{Horak:WeightedLaplcian}
{\sc Danijela Horak and J\"{u}rgen Jost}, {\em Spectra of combinatorial
  {L}aplace operators on simplicial complexes}, Adv. Math., 244 (2013),
  pp.~303--336.

\bibitem{Mulas:Oriented}
{\sc Jürgen Jost and Raffaella Mulas}, {\em Hypergraph laplace operators for
  chemical reaction networks}, Advances in Mathematics, 351 (2019),
  pp.~870--896.

\bibitem{kirkoff_orig}
{\sc G.~Kirchhoff}, {\em Ueber die auflösung der gleichungen, auf welche man
  bei der untersuchung der linearen vertheilung galvanischer ströme geführt
  wird}, Annalen der Physik, 148 (1847), pp.~497--508.

\bibitem{Krivelevich:EML}
{\sc Michael Krivelevich and Benny Sudakov}, {\em Pseudo-random Graphs},
  Springer Berlin Heidelberg, Berlin, Heidelberg, 2006, pp.~199--262.

\bibitem{li2010random}
{\sc Yanhua Li and Zhi-Li Zhang}, {\em Random walks on digraphs, the
  generalized digraph laplacian and the degree of asymmetry}, in International
  Workshop on Algorithms and Models for the Web-Graph, Springer, 2010,
  pp.~74--85.

\bibitem{li2012digraph}
\leavevmode\vrule height 2pt depth -1.6pt width 23pt, {\em Digraph laplacian
  and the degree of asymmetry}, Internet Mathematics, 8 (2012), pp.~381--401.

\bibitem{lim2020hodge}
{\sc Lek-Heng Lim}, {\em Hodge laplacians on graphs}, Siam Review, 62 (2020),
  pp.~685--715.

\bibitem{lu2013high}
{\sc Linyuan Lu and Xing Peng}, {\em High-order random walks and generalized
  laplacians on hypergraphs}, Internet Mathematics, 9 (2013), pp.~3--32.

\bibitem{mohar2020new}
{\sc Bojan Mohar}, {\em A new kind of hermitian matrices for digraphs}, Linear
  Algebra and its Applications, 584 (2020), pp.~343--352.

\bibitem{morbidi2013deformed}
{\sc Fabio Morbidi}, {\em The deformed consensus protocol}, Automatica, 49
  (2013), pp.~3049--3055.

\bibitem{mulas2021spectral}
{\sc Raffaella Mulas and Dong Zhang}, {\em Spectral theory of laplace operators
  on oriented hypergraphs}, Discrete mathematics, 344 (2021), p.~112372.

\bibitem{nikiforov2009cut}
{\sc Vladimir Nikiforov}, {\em Cut-norms and spectra of matrices}.
\newblock 2009.

\bibitem{UGC}
{\sc Prasad Raghavendra and David Steurer}, {\em Graph expansion and the unique
  games conjecture}, in Proceedings of the Forty-Second ACM Symposium on Theory
  of Computing, STOC '10, New York, NY, USA, 2010, Association for Computing
  Machinery, p.~755–764.

\bibitem{restrictedUGC}
{\sc Prasad Raghavendra, David Steurer, and Madhur Tulsiani}, {\em Reductions
  between expansion problems}, in 2012 IEEE 27th Conference on Computational
  Complexity, 2012, pp.~64--73.

\bibitem{Rodriguez:hypergraph}
{\sc J.A. Rodríguez}, {\em Laplacian eigenvalues and partition problems in
  hypergraphs}, Applied Mathematics Letters, 22 (2009), pp.~916--921.

\bibitem{schaub2020random}
{\sc Michael~T Schaub, Austin~R Benson, Paul Horn, Gabor Lippner, and Ali
  Jadbabaie}, {\em Random walks on simplicial complexes and the normalized
  hodge 1-laplacian}, SIAM Review, 62 (2020), pp.~353--391.

\bibitem{trevisian:blog}
{\sc Luca Trevisian}, {\em Lecture notes -- cs294: Graph partitioning,
  expanders and spectral methods}.
\newblock \url{https://lucatrevisan.github.io/expanders2016/lecture04.pdf}.

\bibitem{vanderholst1995short}
{\sc Hein Vanderholst}, {\em A short proof of the planarity characterization of
  colin de verdiere}, Journal of Combinatorial Theory, Series B, 65 (1995),
  pp.~269--272.

\bibitem{von2007tutorial}
{\sc Ulrike Von~Luxburg}, {\em A tutorial on spectral clustering}, Statistics
  and computing, 17 (2007), pp.~395--416.

\bibitem{wocjan2013new}
{\sc Pawel Wocjan and Clive Elphick}, {\em New spectral bounds on the chromatic
  number encompassing all eigenvalues of the adjacency matrix}, The Electronic
  Journal of Combinatorics, 20 (2013), p.~P39.

\bibitem{zhang2021magnet}
{\sc Xitong Zhang, Yixuan He, Nathan Brugnone, Michael Perlmutter, and Matthew
  Hirn}, {\em Magnet: A neural network for directed graphs}, Advances in neural
  information processing systems, 34 (2021), pp.~27003--27015.

\bibitem{zhou2006learning}
{\sc Dengyong Zhou, Jiayuan Huang, and Bernhard Sch{\"o}lkopf}, {\em Learning
  with hypergraphs: Clustering, classification, and embedding}, Advances in
  neural information processing systems, 19 (2006).

\bibitem{Zhou:hypergraph}
{\sc Dengyong Zhou, Jiayuan Huang, and Bernhard Sch\"{o}lkopf}, {\em Learning
  with hypergraphs: Clustering, classification, and embedding}, in Advances in
  Neural Information Processing Systems, B.~Sch\"{o}lkopf, J.~Platt, and
  T.~Hoffman, eds., vol.~19, MIT Press, 2006.

\end{thebibliography}
\end{document}